\documentclass{amsart}
\usepackage{amsmath,amsfonts,amssymb,amscd,pb-diagram,mathrsfs,amsthm,stmaryrd}
\usepackage{cite}
\usepackage[T1]{fontenc}
\def\fg{\mathfrak{g}}
\def\fh{\mathfrak{h}}
\def\fp{\mathfrak{p}}
\def\fk{\mathfrak{k}}
\def\fl{\mathfrak{l}}
\def\fa{\mathfrak{a}}
\def\fx{\mathfrak{x}}
\def\fs{\mathfrak{s}}
\def\fu{\mathfrak{u}}
\def\fL{\mathfrak{L}}
\def\fsp{\mathfrak{sp}}

\def\fsl{\mathfrak{sl}}
\def\RR{\mathbb{R}}
\def\ZZ{\mathbb{Z}}
\def\CC{\mathbb{C}}
\def\sM{\mathscr{M}}
\def\sD{\mathscr{D}}

\def\sE{\mathscr{E}}

\def\sG{\mathscr{G}}
\def\GL{\mathsf{GL}}
\def\SL{\mathsf{SL}}

\def\bC{\mathbf{C}}

\def\PP{\mathbb{P}}
\def\fn{\mathfrak{n}}
\def\fM{\mathfrak{M}}
\def\sK{\mathscr{K}}
\def\sL{\mathscr{L}}
\def\OO{\mathscr{O}}
\DeclareMathOperator{\Mor}{\mathrm{Mor}}

\DeclareMathOperator{\Hom}{\mathrm{Hom}}
\DeclareMathOperator{\Frac}{\mathrm{Frac}}

\DeclareMathOperator{\Aut}{\mathrm{Aut}}
\DeclareMathOperator{\End}{\mathrm{End}}
\DeclareMathOperator{\sym}{\mathrm{sym}}
\DeclareMathOperator{\ev}{\mathrm{ev}}
\DeclareMathOperator{\id}{\mathrm{id}}
\DeclareMathOperator{\Ht}{\mathrm{ht}}
\DeclareMathOperator{\gr}{\mathrm{gr}}
\DeclareMathOperator{\ad}{\mathrm{ad}}
\DeclareMathOperator{\Ad}{\mathrm{Ad}}
\DeclareMathOperator{\MC}{\mathrm{MC}}

\DeclareMathOperator{\Der}{\mathrm{Der}}

\def\u{\underline}

\def\Model{\mathbf{Model}}
\def\Germ{\mathbf{Germ}}
\def\GermSym{\mathbf{GermSym}}
\def\Def{\mathbf{Def}}
\def\bMC{\mathbf{MC}}
\def\Sub{\mathbf{Sub}}
\def\Cartan{\mathbf{Cartan}}

\def\FLA{\mathbf{FLA}}
\def\GLA{\mathbf{GLA}}
\def\DGLA{\mathbf{DGLA}}
\def\Symbol{\mathbf{Symbol}}

\theoremstyle{plain}
\newtheorem{lem}{Lemma}
\newtheorem{pro}{Proposition}
\newtheorem{thm}{Theorem}
\newtheorem{cor}{Corollary}
\theoremstyle{definition}
\newtheorem{defn}{Definition}

\title{Homogeneous models of $C_3$ Monge geometries}
\author{Jan Gutt}
\address{Center for Theoretical Physics, Polish Academy of Sciences, Al. Lotnik\'ow 32/46, 02-668 Warszawa}
\begin{document}
\maketitle
\sloppy
\setcounter{tocdepth}{1}
\tableofcontents

\section{Preliminaries}

\subsection{Introduction} 
\emph{Distributions of Monge type} are a class of strongly regular bracket-generating
distributions introduced by I. Anderson, Zh. Nie and P. Nurowski in~\cite{ANN}. Their symbol algebras
prolong to semisimple graded Lie algebras, thus allowing one to associate a \emph{parabolic geometry}
to any given \emph{Monge distribution}. This article is devoted to the classification problem
for homogeneous models of Monge distributions of type $C_3$ in dimension eight. 
These are rank $3$ distributions
in dimension $8$, whose symbol algebra prolongs to the semisimple Lie algebra $\fsp(6,\RR)$
with a suitable grading. Applying the techniques of parabolic geometry, one associates to
every such distribution $\sD$ a pair of invariants (components of the harmonic curvature): 
a `scalar', and a `quintic'. The
scalar invariant is a section of a natural line bundle; the quintic one is
a section of $S^5\sD^*$ twisted by a natural line bundle. We restrict the classification problem to
at least $2$-transitive homogeneous $C_3$ Monge distributions whose scalar invariant vanishes.
The general classification algorithm, as well as most of its application to the particular
problem, are joint work with Ian Anderson (USU Logan); its formalisation in this
paper is a sole responsibility of the present author.

The classification algorithm we use had been developed by Ian Anderson
and the author in the context of classifying homogeneous models of
strongly regular distributions whose symbol prolongs to a semi-simple Lie algebra $\fg$. 
It relies on a reformulation of the classification problem in terms of
filtered deformations of certain graded subalgebras of $\fg$, at the same time
using parabolic geometry to compute the invariants (harmonic curvature) and
symmetries. Necessary techniques of deformation theory and parabolic geometry have been
implemented in the {\tt DifferentialGeometry} package for {\sc Maple}. 
An in-depth technical
presentation will appear as a separate work (in preparation). Here we do not touch
upon these technical details.

On the other hand, a significant amount of space is devoted to showing that the conceptual
foundations of our
classification method are sound. Since we pass through several reformulations
of a classification problem, it becomes crucial to carefully keep track not only of
the objects being classified, but also of the notion of equivalence. While there
are other ways to do it, we subscribe to the view that a categorical formalisation
is the most convenient and reliable bookkeeping device. Hence we use the categorical
language, perhaps to a somewhat larger extent than typical for this area. We hope that,
as intended,
the reader will find it an aid rather than an obstacle.

The present paper is to be regarded as forming a pair with a parallel development
due to Ian Anderson and Pawel Nurowski~\cite{AN}. The latter work approaches the 
same classification problem, although covering a slightly different area: it does
not restrict to at least $2$-transive models, but, on the other hand, considers
only the case where the quintic invariant is a fifth power (i.e., type $N$ below).
Importantly, Anderson and Nurowski use a very different framework, namely the
method of \emph{Cartan reduction}. 
There, a generic Monge geometry is viewed
as an abstract exterior differential system on a principal bundle; one then
exploits the action of the structure group to normalise the structure functions
(curvatures, torsions), gradually reducing the bundle. As each normalisation step
may introduce a number of branches (corresponding to the orbits of the structure
group in some space of structure coefficients), one eventually obtains a
tree whose leaves carry the various families of homogeneous models; tracing
the path form a leaf to the root, one may label these families by the
normalisation conditions.
Thus, in the course of classifying the homogeneous
models, Anderson and Nurowski produce also the invariants for different classes of models,
as well as explicit Cartan connections with strong normalisations. Furthermore, they
provide realisations of Monge geometries in terms of ordinary differential equaitons,
and identify the structure invariants of a Cartan geometry with the differential invariants of an
ODE --
an aspect we do not touch upon in the present paper. We shall comment further on
the relation between the two papers in the final section.

\subsection{Acknowledgements}
I am immensely grateful to Professor Ian Anderson for a rewarding collaboration
and his hospitality during my stay at USU Logan in Spring 2015. The present article
is to a very large degree based on our joint work. I am likewise indebted to Professor
Pawe\l{} Nurowski for inviting me to join the project, providing an excellent working environment,
and generously sharing the details of his own approach to the problem.
I acknowledge the support of the Polish National Science Centre (NCN)
grant DEC-2013/09/B/ST1/01799.

\subsection{Lie-theoretic setup} 
Throughout the paper, we let $\fg \simeq \fsp(6,\RR)$ be the split real form
of the semisimple Lie algebra of type $C_3$. Fixing a Cartan subalgebra $\fh \subset \fg$,
we have the root system $\Phi \subset \fh^*$, a subset of positive roots $\Phi^+ \subset \Phi$
and simple roots $\Delta \subset \Phi^+$. We use Bourbaki labelling $\Delta = \{ \alpha_1,
\alpha_2,\alpha_3\}$ where $\alpha_3$ is long; let $\varpi_1,\varpi_2,\varpi_3 \in \fh^*$
be the corresponding fundamental weights so that 
$\langle \varpi_i, \alpha_j^\vee\rangle = \delta_{ij}$. The height function
$$ \Ht : \Phi \to \ZZ,\quad \Ht(\alpha) = \langle \varpi_2+\varpi_3,\alpha^\vee \rangle. $$
induces a grading
$\fg_\bullet$ such that
$$ \fg_0 = \fh \oplus \bigoplus_{\Ht(\alpha)=0} \fg_\alpha,
\quad
\fg_i = \bigoplus_{\Ht(\alpha) = i} \fg_\alpha,\
i\neq 0. $$
This turns $\fg$ into a graded Lie algebra
$$
\fg = \fg_{-3} \oplus \fg_{-2} \oplus \fg_{-1} \oplus \fg_0 \oplus \fg_1 \oplus \fg_2 \oplus \fg_3.
$$
We also use the filtration $\fg^\bullet$ where $\fg^i = \bigoplus_{j\ge i}\fg_j$,
with a natural identification $\gr \fg^\bullet = \fg_\bullet$. In particular,
$\fp = \fg^0$ is a parabolic subalgebra, $\fp_+ = \fg^1$ its nilpotent radical,
and $\fg_0$ the Levi factor. 

Since $\fg$ is semi-simple, the grading $\fg_\bullet$ is induced by
a unique $E \in \fh$, called the \emph{grading element}, such that $\fg_i$ is the eigenspace of
$\ad_E$ with eigenvalue $i$. This allows us to induce a compatible grading on
any finite-dimensional $U(\fh)$-module $V$, defining $V_i$ to be the eigenspace
of $E\cdot$ with eigenvalue $i$. Again, we $V^i = \bigoplus_{j\le i} V_j$ so that
$V^\bullet$ is a filtration such that $\gr V^\bullet = V_\bullet$ naturally
(as $U(\fh)$-modules). If $V$ is a representation of $\fg_0$, 
the grading $V_\bullet$ and the filtration $V^\bullet$
are  $\fg_0$-equivariant. If $V$ is a representation
of $\fp$, the filtration $V^\bullet$ is $\fp$-equivariant.

\subsection{Monge distributions} 

We let $\fg_- = \bigoplus_{i<0}\fg_i$,
a graded nilpotent Lie algebra. As is shown in [ANN],
$\fg$ is precisely the
Tanaka prolongation of $\fg_-$. In particular, $\fg_0$
is the Lie algebra of derivations of $\fg_-$; we define
$G_0 = \Aut \fg_-$. Note that $G_0$ acts naturally
on $\fg$, and the action preserves the grading.

The ranks of the relevant graded subspaces are as follows:
$$
\dim \fg_{-3} = 3,\ \dim \fg_{-2} = 2,\ \dim \fg_{-1} = 3,\ \dim \fg_0 = 5.
$$
Following op. cit., we have $G_0$-equivariant identifications
\begin{eqnarray*}
\begin{split}\fg_{-1} &= \fx \oplus \fa \\
\fg_{-2} &= \fx \otimes \fa \\
\fg_{-3} &= \fx \otimes S^2\fa \end{split}
\qquad
\begin{split} \fg_0 &= \End \fx \oplus \End \fa \\
G_0 &= \GL(\fx) \times \GL(\fa)\end{split}
\end{eqnarray*}
where $\fx$ is the root subspace of $-\alpha_3$,
$\fa$ is a rank $2$ abelian subalgebra in $\fg_-$,
and the Lie bracket is given by the natural
$G_0$-equivariant projections.

\begin{defn}
A \emph{$C_3$ Monge distribution} is a strongly regular rank $3$ distribution
on an $8$-dimensional manifold, whose symbol is isomorphic to $\fg_-$.
\end{defn}

\subsection{Associated parabolic geometries} \label{ss:para}

Let $P_+$ be the connected,
simply-connected unipotent Lie group with Lie algebra $\fp_+$, and set
$P = G_0 \ltimes P_+$ with respect to the natural $G_0$-action on $\fp_+\simeq\fg_-^*$.
\begin{pro}
\label{pro:equiv-para}
There is an equivalence of categories between:
\begin{enumerate}
\item the category of $C_3$ Monge distributions and local equivalences,
\item the category of regular, normal parabolic geometries of type $(\fg,P)$,
and local equivalences.
\end{enumerate}
\end{pro}
\begin{proof}
This follows from Theorem 3.1.14 of~\cite{CS}, where we refer to~\cite{ANN} for
the fact that $H^1(\fg_-,\fg)^1=0$.
\end{proof}
Here a \emph{local equivalence} between a pair of distributions $\sD \subset TM$
and $\sE \subset TN$ is a local diffeomorphism $f:M \to N$
such that $f^{-1}\sE = \sD$. A \emph{parabolic geometry} of type $(\fg,P)$
is a $P$-principal bundle $\pi:\sG \to M$ together with a $\fg$-valued Cartan connection $\omega$
inducing an isomorphism of $\pi^*TM$ with the trivial bundle with fibre $\fg/\fp$.
It is \emph{regular} if its curvature function $\kappa:\sG \to \Lambda^2\fp_+\otimes\fg$,
factors through the degree $1$ subspace for the $\fp$-equivariant filtration induced
by the grading element.
If is \emph{normal} if the curvature function factors through the space of cycles
$Z_2(\fp_+,\fg)$
for the natural codifferential on the chain complex $C_\bullet(\fp_+,\fg)$ computing
Lie algebra homology of $\fp_+$ with values in $\fg$ (using the adjoint representation).
Finally, a \emph{local equivalence} between a pair of parabolic geometries
$(\sG/M, \omega)$ and $(\sG'/N, \omega')$ is a local diffeomorphism
$f:M\to N$ together with a $P$-principal bundle isomorphism
$\tilde f: \sG \to f^*\sG'$ over $M$ such that the induced map $\sG \to \sG'$
pulls back $\omega'$ to $\omega$.

We refer to~\cite{CS} for the necessary background on Cartan and in particular
parabolic geometries. Let us recall that viewing the curvature function
of a regular, normal parabolic geometry as a function
$$ \kappa : \sG \to Z_2(\fp_+,\fg)^1 $$
we may consider its composite with projection to homology: that is the \emph{harmonic
curvature}
$$ \kappa_H : \sG \to H_2(\fp_+,\fg)^1 $$
where the upper index $\cdot^1$ refers to the $\fp$-equivariant filtration
induced by the grading element. The homology space is a completely reducible
representation of $\fp$, so that $\fp_+$ acts trivially, and the action factors through
$\fg_0$. The celebrated theorem of Kostant allows one to find the highest weights of 
the simple $U(\fg_0)$-submodules. One then finds that
$$
H_2(\fp_+,\fg)^1 \simeq \RR[-10,4] \oplus S^5\fa^* \otimes \RR[-3,1]
$$
where $\RR[p,q] \simeq \fx^{-q} \otimes (\det\fa)^{\frac{q-p}{2}}$ 
is a one-dimensional
representation of $G_0$ (provided $p+q$ is even). In particular restricting to the semisimple part
of $G_0$, isomorphic to $\SL(2,\RR)$, the two summands correspond to
scalars and quintic binomials, respectively. Accordingly, the harmonic curvature
$\kappa_H$ decomposes into what we referred to as the scalar and quintic component.

Observe that the harmonic curvature as an $H_2(\fp_+,\fg)^1$-valued
function on $\sG$ factors through the $G_0$-principal bundle $\sG_0 = \sG/P_+$.
Being $P$-equivariant, it may then be viewed as a $G_0$-equivariant
map $\sG_0 \to H_2(\fp_+,\fg)^1$, and thus a section of the associated
vector bundle over $M$. Since $\sG_0$ may be identified with an adapted
frame bundle for $\sD$, it follows that the associated bundle
$\sG_0\times^{G_0} H_2(\fp_+,\fg)^1$ is a certain tensor bundle; thus,
$\kappa_H$ may be viewed as a tensorial invariant defined on $M$.

\subsection{Classification problem} 
We may now state the classification problem in a precise way:
\emph{classify local equivalence classes of Monge $C_3$ distributions such that
(1) the distribution admits a transitive symmetry group
with at least two-dimensional 
isotropy, (2) the scalar component of the harmonic curvature of the corresponding
parabolic geometry of type $(\fg,P)$ vanishes.}

We follow the idea that an equivalence problem (including auto-equivalences, i.e. symmetries)
is naturally organised into a groupoid, i.e. a small
category whose arrows are isomorphisms (without the smallness assumption,
we have a `large groupoid'). 
The classification problem amounts to
describing the set of isomorphism classes of objects, and the latter
is invariant under equivalence of categories. In particular, even if
the category we begin with is large, we may pass to an equivalent
small gropoid.

In the case at hand, we shall set up such framework for a slightly more general problem, namely
that of classifying \emph{all} locally homogeneous Monge $C_3$ distributions 
up to local equivalence.
We begin with the following category of \emph{pointed} locally homogeneous models.
\begin{defn}$\Model$ is the category whose objects are
triples $(M, \sD, m)$ where $\sD \subset TM$ is a $C_3$ Monge distribution
with an \emph{infinitesimally transitive} symmetry algebra, and
$m \in M$ is a point. Morphisms $(M,\sD,m) \to (N,\sE,n)$ are
local diffeomorphisms $f:M \to N$ such that $f^{-1}\sE = \sD$ and $f(m)=n$. 
\end{defn}
Thus defined, $\Model$ is not a (large) groupoid, but we can turn it
into one by formally inverting its morphisms.
Recall that given a category $\bC$, its \emph{localisation} at
the collection $\Mor\bC$ of all morphisms is the large groupoid
$\Frac\bC = \bC[(\Mor\bC)^{-1}]$, together with a functor $\bC \to \Frac\bC$,
satisfying a universal property with respect to functors from $\bC$ to
large groupoids. If $\bC$ has pullbacks,
one may represent the arrows $a \to c$ of $\Frac\bC$ by diagrams 
$a \leftarrow b \rightarrow c$ in $\bC$
(two such diagrams define
the same arrow if they both `factor' a third such diagram).
Note that $\Model$ does have pullbacks, whence
$\Frac\Model$ is simply the category of triples $(M,\sD,m)$
where a morphism to $(N,\sE,n)$ is a diagram $M\leftarrow \tilde M \to N$
of pointed local equivalences. 

It is intuitively clear
that, having inverted pointed local diffeomorphisms, we might as well work
with germs of distribution at a point $o$ of the fixed manifold $\RR^8$.
These are defined in the usual way as equivalence classes $\sD$ of pairs $(U,\sD_U)$ where
$U$ is an open neighbourhood of $o$ and $\sD_U \subset TU$ a distribution. The 
\emph{symmetry algebra} $\sym \sD$ is then understood to be the local symmetry algebra
at $o$ of any representative, viewed as a Lie algebra of germs of vector fields.
An \emph{equivalence} of germs $\sD$ and $\sE$ is a germ at $o$ of a diffeomorphism
$\RR^8 \to \RR^8$ sending $o$ to itself and $\sD$ to $\sE$.

\begin{defn}$\Germ$ is the groupoid whose objects are germs at $o$ of Monge $C_3$
distributions $\sD$ on $\RR^8$ such that the evaluation map $\ev_o : \sym\sD \to T_o\RR^8$
is surjective, and whose morphisms are equivalences as defined above. 
\end{defn} 

\begin{pro}\label{pro:equiv-germ}
There is an equivalence of categories between $\Germ$ and $\Frac\Model$.
\end{pro}

The consequence is that the local equivalence problem for locally homogeneous Monge $C_3$
distributions is encoded by the small groupoid $\Germ$. In particular, the original
classification problem stated in the beginning of this subsection reduces to the
problem of describing isomorphism classes of objects of a suitable sub-groupoid
of $\Germ$ (consisting of germs $\sD$ with 
$\ker\ev_o \ge 2$ and vanishing scalar component of the germ of harmonic curvature at $o$).

\begin{proof}
For each object $\sD$ of $\Germ$,
note that $\dim\sym\sD \le 21$ so that we may, and do, choose a connected, simply-connected
Lie group $K_\sD$ together with a Lie algebra isomorphism $\fk_\sD \to \sym\sD$.
Then, let $L_\sD \subset K_\sD$ be the subgroup whose Lie algebra corresponds to
the isotropy algebra in $\sym\sD$, and such that $K_\sD/L_\sD$ is simply-connected.
The infinitesimal action
of $\sym\sD$ on a small neighbourhood of $o \in \RR^8$
induces a germ of a local diffeomorphism from $K_\sD/L_\sD$
to $\RR^8$, mapping the origin to $o$, and
pulling back $\sD$ to the germ of a $K_\sD$-equivariant
distribution $\tilde\sD$ on $K_\sD/L_\sD$.
It is then straightforward to check that mapping $\sD$
to $(K_\sD/L_\sD, \tilde\sD)$ extends to a functor
$F_0:\Germ \to \Model$. 

In the opposite direction, we choose for each $(M,\sD,m)$ in $\Model$
a diffeomorphism $f : U \to \RR^8$ from an open neighbourhood $U$ of $m \in M$
to an open subset of $\RR^8$ such that $f(m)=o$. Then $\bar\sD=[(f(U), f_*\sD)]$ defines
an object of $\Germ$, and it is again straightforward to check that
it extends to a functor $G_0 : \Model \to \Germ$. 

Let $I : \Model \to \Frac\Model$ be the canonical embedding.
Since $\Germ$ is a groupoid, it follows that $G_0$ factors uniquely through $I$
so that $G_0 = GI$. Set $F = IF_0$.
We now need
too check that the pair $G$, $F$ forms an equivalence of
categories. 
First, we have a natural isomosphism $\id \to G_0F_0 = GF$ whose component at $\sD$
is the germ of a local diffeomorphism $\RR^8 \to \RR^8$ arising as the composite of
germs of diffeomorphisms $\RR^8 \to K_\sD/L_\sD \to \RR^8$ arising in the definition
of $F_0$, $G_0$. 
On the opposite side,
there is a natural isomorphism $\id\to FG$ whose component at
$(M,\sD,m)$ is given by the diagram $M \leftarrow U \to K_{\bar\sD}/L_{\bar\sD}$ 
where the rightmost map is a representative of the
composite of germs at $m$ of local diffeomorphisms
$M\to\RR^8\to K_{\bar\sD}/L_{\bar\sD}$ arising in the definition of $G_0$, $F_0$.
\end{proof}

As a side remark, note that the category of parabolic geometries (of a given type)
and local equivalences is enriched over the category of finite-dimensional
smooth manifolds. Using Proposition \ref{pro:equiv-para},
the same holds for $\Model$ and $\Frac \Model \approx \Germ$, whence
in particular $\Aut\sD$ for $\sD$ in $\Germ$ is naturally
a Lie group: the isotropy group at $o\in\RR^8$. It follows that
the isotropy condition appearing in our original classification problem
corresponds $\dim\Aut\sD \ge 2$. We shall not pursue this interpretation.

\section{From distributions to deformations}
\label{sec:dide}

\subsection{Introduction}
The central idea of our approach to the problem of classification of homogeneous
models is to reformulate it in terms of the deformation theory of filtered Lie algebras.
The way we present it in this section applies to arbitrary strongly regular
distributions whose symbol ($\fg_-$ in our case)
has a finite-dimensional Tanaka prolongation ($\fg$ in our case): indeed, these
are the only properties of $(\fg_-,\fg)$ we shall refer to in what follows.

Our first step will be a pretty standard trick: replace the property of `there exists
a transitive symmetry algebra' with the \emph{datum} of such algebra. That is,
we will consider germs of Monge $C_3$ distributions $\sD$ together with an explicit
transitive algebra of germs of vector fields $\fk \subset \sym\sD$. Now, classifying
pairs $(\sD,\fk)$ up to equivalence yields a parially ordered set rather than a set,
where the order relation reflects inclusions between the algebras. Following the
well-known dictum `a groupoid is the categorification of a set; a category is
the categorification of a poset', we shall organise the corresponding equivalence problem
into a category rather than a groupoid.

\subsection{Symmetries and the symbol}
As a preliminary step, let us recall a basic fact about local symmetry algebras of distributions
(see e.g.~\cite{Kruglikov}). 
Let $\FLA$, resp. $\GLA$, denote the category of finite-dimensional filtered,
resp. graded, Lie algebras and filtration-, resp. grading-preserving homomorphisms. The associated 
graded construction gives a functor $\gr : \FLA \to \GLA$. 
Denote by $\Symbol \subset \GLA$ the subcategory whose objects are
negatively graded algebras, generated in degree $-1$, with \emph{finite-dimensional Tanaka prolongation},
and whose morphisms are graded Lie algebra isomorphisms.
We then have
the Tanaka prolongation functor
$$ \Pr : \Symbol \to \GLA $$
together with a natural monomorphism $\id \to \Pr$ (of functors $\Symbol\to\GLA$).

Given a germ of a bracket-generating distribution $\sD$ at $o \in \RR^8$, recall
that $\sD$ induces a germ of a filtration on the tangent bundle,
and thus in particular a filtration $T_o^\bullet\RR^8$ on the tangent space at $o$.
It is the weakest filtration compatible with the Lie bracket
of vector fields, and such that its degree $-1$, resp. $0$, sub-bundle is
precisely $\sD$, resp. zero. It follows that the Lie bracket turns the
associated graded $\gr T_o\RR^8$ into a graded nilpotent Lie algebra $\sigma_\bullet(\sD)$.
This is easily seen to produce a functor
$$ \sigma : \Germ \to \Symbol. $$
\begin{lem}
The assignment $\sD \mapsto \sym\sD$ extends naturally to a functor $\sym : \Germ \to \FLA$
together with a natural monomorphism $\sigma \to \gr\sym \to \Pr\sigma$ of functors
$\Germ \to \GLA$. 
\end{lem}
\begin{proof}
Let $\sD$ be an object of $\Germ$.
We first need to exhibit the filtration on $\sym\sD$. 
We set $\sym^i\sD = \ev_o^{-1} T_o^i\RR^8$ for $i<0$, and then 
let
$$
\sym^i\sD = \{ X \in \sym\sD\ |\ \ev_o [\sD,[\cdots[\sD,X]]]=0\ \textrm{($i$ copies)} \}
$$
for $i\ge0$. It is straightforward to check functoriality. The inclusion
$\sigma(\sD) \to \gr\sym\sD$ is clear by construction: indeed, the symbol
is identified with the negative part of $\gr\sym\sD$. Then the homomorphism
$\gr\sym\sD \to \Pr\sigma(\sD)$ arises from the universal property of Tanaka
prolongation with respect to graded Lie algebras extending the symbol in 
non-negative degrees. Naturality of the two maps is again straightforward to check.

It remains to verify injectivity of $\gr_i\sym\sD \to \Pr\sigma(\sD)_i$ for all $i$.
Again, for $i<0$ this follows directly from the construciton.
Letting $\sD_o$ be the fibre of $\sD$ at $o \in \RR^8$,
we then have for each $i\ge0$ a commutative diagram
$$\begin{diagram}
\node{\sym^i\sD} \arrow{e}\arrow{se} \node{\gr_i\sym\sD} \arrow{s} \arrow{e} \node{\Pr\sigma(\sD)_i} \arrow{s} \\
\node{} \node{\Hom(\bigotimes^{i+1} \sD_o, \sD_o)}\arrow{e,t}{\simeq} 
\node{\Hom(\bigotimes^{i+1} \sigma_{-1}(\sD), \sigma_{-1}(\sD))}
\end{diagram}
$$
where the vertical arrows are induced by the Lie bracket of vector fields (on the left),
resp. the adjoint representation (on the right). By the defining properties of the Tanaka
prolongation, the right vertical arrow is injective, so that it is enough to check
that the left one is. Suppose thus $X \in \sym^i\sD$ is such that $\bar X \in \gr_i\sym\sD$
is mapped to $0:\bigotimes^{i+1}\sD_o \to \sD_o$; then $X \in \sym^{i+1}\sD$ by definition
of the filtration,
whence $\bar X=0$. 
\end{proof}

\subsection{Distributions and deformations}
As we have indicated in the introductory subsection, we shall proceed by
adding a transitive Lie algebra of symmetries (not necessarily the maximal one)
as a datum in the classification problem.
\begin{defn}
$\GermSym$ is the category whose objects are pairs $(\sD, \fk)$ such that
$\sD$ is an object of $\Germ$, and $\fk \subset \sym\sD$ is a Lie subalgebra
such that $\ev_o : \fk \to T_o\RR^8$ is surjective; its morphisms from
$(\sD,\fk)$ to $(\sD', \fk')$ are those morphisms from $\sD$ to $\sD'$ in $\Germ$
whose underlying germ of a diffeomorphism $\RR^8 \to \RR^8$ maps
$\fk$ into $\fk'$.
\end{defn}

\begin{lem}\label{lem:germsym}
The obvious forgetful functor $\GermSym \to \Germ$ admits a full and faithful right adjoint.
\end{lem}
\begin{proof}
The right adjoint sends $\sD$ in $\Germ$ to $(\sD, \sym\sD)$ in $\GermSym$,
and its action on morphisms is given by functoriality of $\sym$. 
Indeed, morphisms $(\sD, \fk) \to (\sD',\sym\sD')$ in $\GermSym$
are the same as morphisms $\sD \to \sD'$ in $\Germ$. 
Taking $\fk = \sym\sD$, we also find that this right adjoint is
full and faithful. 
\end{proof}

The adjunction $\Germ \rightleftarrows \GermSym$ turns $\Germ$
into a \emph{reflexive} subcategory of $\GermSym$. In particular,
we have the composite functor $S : \GermSym \to \GermSym$, sending
$(\sD,\fk)$ to $(\sD, \sym\sD)$. Then
$\Germ$ is precisely the full subcategory 
consisting of objects preserved by $S$; obviously that is merely 
an elaborate way to say $\fk = \sym\sD$. This way we have embedded
the groupoid controlling our classification problem into a larger (but still small!)
category. Our strategy now will be to describe the set of isomorphism
classes of objects of $\GermSym$, and only later to check which ones
are preserved by $S$. The main reason for admitting this seemingly spurious
wealth of objects is that we may now completely forget about distributions!

\begin{lem}
The obvious forgetful functor $\GermSym \to \FLA$ is a full and faithful
embedding. Its essential image is the full subcategory of filtered Lie algebras
$\fk$ admitting a graded Lie algebra monomorphism $\gr\fk \to \fg$ whose image contains $\fg_-$.
\end{lem}
\begin{proof}
Let $\bC \subset \FLA$ denote the full subcategory described in the statement
of the Lemma. Clearly, the forgetful functor factors through $A : \GermSym \to \bC$. We shall
construct an essential inverse $B:\bC \to \GermSym$. First, for each
$\fk$ in $\bC$ choose a connected, simply connected Lie group $K$ together with
an identification of its Lie algebra with $\fk$. Let $K^0 \subset K$ be
the subgroup with Lie sub-algebra $\fk^0$ such that $K/K^0$ is simply-connected.
Let $\sD_\fk$ be the $K$-invariant distribution on $K/K^0$ corresponding
to $\fk^{-1}/\fk^0$. Since $\fk$ is in $\bC$, it follows that $(K/K^0, \sD_\fk)$
is an object of $\Model$. Choose a germ $f_\fk$ of a diffeomorphism $K/K^0 \to \RR^8$
sending the origin to $o$ (cf. the construction of $G_0$ in the proof of Proposition
\ref{pro:equiv-germ}). Finally let $B(\fk) = (f_{\fk*}\sD_\fk, f_{\fk*}\fk)$
where $\fk$ is viewed as a Lie algebra of vector fields on $K/K^0$.
It is straightforward to check that $\fk \to B(\fk)$
extends to a functor. Now, $AB\simeq\id_{\bC}$ by construction.
On the other hand, the natural isomorphism $BA \simeq \id_{\GermSym}$
is given on $(\sD,\fk)$ by the germ of a diffeomorphism
$K/K^0 \to \RR^8$ integrating the $\fk$-action on both sides.
\end{proof}

We apply the well-known trick once again and replace the property `there exists...'
with explicit data.
\begin{defn}
$\Def$ is the category whose objects are pairs $(\fk, \iota)$
where $\fk$ is a filtered Lie algebra, and $\iota : \gr\fk \to \fg$ is
a graded Lie algebra monomorphism such that $\fg_- \subset \iota(\gr\fk)$.
Its morphisms from $(\fk,\iota)$ to $(\fk',\iota')$ are
pairs $(\varphi,g)$ where
$\varphi:\fk\to\fk'$ is a filtered Lie algebra homomorphism,
and $g \in G_0$ an element such that
$\iota'\circ\gr\varphi = \iota \circ \Ad_g$.
\end{defn}

\begin{lem}\label{lem:g-from-phi}
Let $(\fk,\iota)$ and $(\fk',\iota')$ be objects of $\Def$.
Then for each filtered Lie algebra homomorphism $\varphi:\fk\to\fk'$
there exists a unique $g\in G_0$ such that $(\varphi,g)$ is
a morphism $(\fk,\iota) \to (\fk',\iota')$ in $\Def$.
\end{lem}
\begin{proof}
Note that $\gr\varphi$ defines a graded Lie algebra automorphism
$$
\fg_- \xrightarrow{\iota^{-1}} \fk_- \xrightarrow{\gr\varphi} \fk'_-
\xrightarrow{\iota'} \fg_-
$$
and thus an element $g \in G_0$. By the universal property of Tanaka prolongation,
we then have $\iota' \circ\gr\varphi = \iota\circ\Ad_g$. On the other hand, every
element $g \in G_0$ satisfying the latter equation induces the same automorphism
of $\fg_-$, whence uniqueness.
\end{proof}

\begin{lem}\label{lem:gs-def}
The obvious forgetful functor $\Def \to \FLA$
is a full and faithful embedding onto the essential image of
the other obvious forgetful functor $\GermSym \to \FLA$.
As a consequence, there is an equivalence of categories between
$\Def$ and $\GermSym$.
\end{lem}
\begin{proof}
Once again we use the notation $\bC \subset \FLA$ for the full subcategory
in question. The forgetful functor factors through $A:\Def\to\bC$ by definition.
Its essential inverse $B:\bC\to\Def$ is constructed by choosing
for each $\fk$ in $\bC$ a graded Lie algebra monomorphism
$\iota_\fk : \gr\fk\to\fg$ such that $B(\fk) = (\fk,\iota_\fk)$ is an object of $\Def$.
Given a filtered homomorphism $\varphi : \fk \to \fk'$ in $\bC$,
let $g_\varphi \in G_0$ be the unique element such that
$(\varphi,g_\varphi)$ is a morphism $B(\fk)\to B(\fk')$ in $\Def$ (cf. Lemma \ref{lem:g-from-phi}).
Set $B(\varphi) = (\varphi,g_\varphi)$.
The  natural isomorphisms $AB \simeq \id_\bC$ and $BA \simeq \id_\Def$ are tautological:
the component of the former at $\fk$ is $\id_\fk$ as a morphism in $\bC$, 
while the component of the latter at $(\fk,\iota)$ is $\id_\fk$ as a morphism 
to $(\fk,\iota_\fk)$ in $\Def$.
\end{proof}

Our last step in this sub-section is to switch focus from the filtered
Lie algebra $\fk$ to the image of $\iota$, a graded subalgebra $\u\fk\subset\fg$ (we
will use this notational convention throughout the paper).
\begin{defn}
$\Sub$ is the category whose objects are graded sub-algebras of $\fg$ containing
$\fg_-$, and whose morphisms from $\u\fk$ to $\u\fk'$ are elements $g \in G_0$ such that
$\Ad_g \u\fk \subset \u\fk'$.
\end{defn}
The reason for including an element of $G_0$ explicitly in the definition of
morphisms in $\Def$ is that we now have a functor
$\Def \to \Sub$ sending $(\fk,\iota)$ to $\iota(\gr\fk)$ and
$(\varphi,g)$ to $g$. We thus view $\Def$ as a category over $\Sub$.
Recall that the \emph{fibre} $\Def_{\u\fk}$ of $\Def$ over an object $\u\fk$ of $\Sub$
is the sub-category of $\Def$ whose objects are mapped to $\u\fk$ in $\Sub$,
and whose morphisms are mapped to $\id_{\u\fk}$ in $\Sub$. In particular, in our case
$\Def_{\u\fk}$ is a large groupoid: indeed, a filtered Lie algebra homomorphism $\varphi:\fk \to \fk'$ is
an isomorphism if (and only if) $\gr\varphi$ is.
Now, given $(\fk,\iota)$ in $\Def_{\u\fk}$, we may view $\fk$ as a \emph{filtered deformation}
of $\u\fk$, trivial under passage to the associated graded.
We shall thus use
deformation theory to study the fibres of $\Def \to \Sub$.

\subsection{The formalism of DGLAs}

Deligne's principle states that every reasonably well-behaved deformation
problem (in char. $0$) is controlled by a \emph{differential Graded Lie algebra}
(abbreviated DGLA). That is, the equivalence problem for deformations
of an algebraic or geometric object (in our case encoded in the large groupoid $\Def_{\u\fk}$)
may be replaced with a standard equivalence problem associated with a
DLGA. One may then consider the question of existence of a `deformation space'
abstracting from the contingencies of the original objects (in our case, this
at the very least allows us to avoid unnecessarily cluttered notation). 
This subsection reviews the relevant notions and constructions without
any reference to the remaining parts of the paper. We shall resume the main
narrative in the next subsection. In part, we adapt the presentation of~\cite{Manetti}.

It is difficult to avoid a slight terminological inconsistency in
the use of the word `graded Lie algebra'. Thus far, it denoted a Lie algebra
together with a grading compatible with the Lie bracket. On the other hand,
in the context of DGLA, the term denotes essentially what is called a Lie super-algebra.
We shall use a capitalised `Graded' for the latter meaning; furthermore,
Graded degree will appear as an upper index (as opposed to graded degree,
appearing as a lower index). That is, a Graded vector space is $V = \bigoplus_p V^p$;
given a homogeneous element $v \in V^p$, we set $|v|=p$. 
\begin{defn}\label{def:dgla}
A DGLA is a Graded vector space $\fL = \bigoplus_p \fL^p$ together with:
\begin{enumerate}
\item a bracket $[,] : \fL \otimes \fL \to \fL$ of degree $0$,
\item a differential $d : \fL \to \fL$ of degree $1$, $d^2=0$,
\end{enumerate}
satisfying:
\begin{enumerate}
\item Graded skewness $[x,y] = (-1)^{|x||y|+1} [y,x]$,
\item Graded Jacobi identity $(-1)^{|x||z|}[x,[y,z]] +\mathrm{cycl.} = 0$,
\item Graded Leibniz rule $d[x,y] = [dx,y] + (-1)^{|x|} [x,dy]$
\end{enumerate}
for homogeneous elements $x,y,z \in \fL$.
\end{defn}

Note that a DGLA is in particular
a cochain complex. We define the cycles $Z^p(\fL)$, boundaries $B^p(\fL)$ and cohomology $H^p(\fL)
= Z^p(\fL)/B^p(\fL)$ in the usual way. 
Observe also that $\fL^0$ is a (usual) Lie algebra acting
on $\fL$ by Graded derivations of degree $0$. 
If this adjoint action is nilpotent, we consider the 
connected, simply connected Lie group $\exp\fL^0$ 
and its induced `adjoint' action by automorphisms of $\fL$, denoted $\Ad$. 
We may then identify $\exp\fL^0$ with $\fL^0$ as a manifold, so that
the exponential map becomes the identity, while the multiplication
and inverse, as well as $\Ad$, are given by polynomial maps.

The `standard' equivalence
problem associated with a DGLA is expressed by the following notions.
\begin{defn}
Let $\fL$ be a finite-dimensional DGLA. 
\begin{enumerate}
\item An element $x \in \fL^1$ is \emph{Maurer-Cartan} if $$ dx + \frac{1}{2}[x,x] = 0. $$
The algebraic subset of Maurer-Cartan elements is denoted $\MC(\fL) \subset \fL^1$.
\item The \emph{infinitesimal gauge action} is an affine action of $\fL^0$ on $\fL^1$
defined by
$$ \ast:\fL^0 \times \fL^1 \to \fL^1,\quad y\ast x = [y,x] - dy. $$
If the adjoint action of $\fL^0$ on $\fL$ is nilpotent, we exponentiate the above to
the \emph{gauge action} of $\exp\fL^0$ on $\fL^1$.
Together with this action, $\exp\fL^0$ is the \emph{gauge group} of $\fL$.
\end{enumerate}
\end{defn}
It is a simple exercise to check that the gauge action of $\exp\fL^0$ preserves
$\MC(\fL) \subset \fL^1$. We may thus consider the question of gauge-equivalence
of Maurer-Cartan elements.
\begin{defn}
Under the above assumptions, the 
\emph{gauge action groupoid} $\MC(\fL) \sslash \exp\fL^0$ has $\MC(\fL)$ as the set of objects,
and $$\Hom(x,y) = \{ u \in \exp\fL^0\ |\ u\ast x=y \}$$
as homsets for $x,y\in\MC(\fL)$.
\end{defn}

Observe that the gauge action is affine. In fact, one may consider
an extension $\fL \oplus \langle d \rangle$ of the original DGLA
obtained by formally adjoining an element $d$ in Graded degree $1$ with the obvious
relations $dd=0$, $[d,x]=dx$. The adjoint action of $\exp\fL^0$
extends naturally to an action on the extended DGLA, preserving the affine
subspace $d+\fL$.
Then, identifying $\fL$ with the affine subspace
$d + \fL$, one checks that the gauge action on $\fL$ corresponds to the
naturally extended adjoint action, restricted to $d + \fL$. In
symbols, $\Ad_u (d+x) = d+u\ast x$.

An ideal solution to the classification problem would be to construct
the `deformation space' $\MC(\fL) / \exp\fL^0$
as a manifold or variety. In general, this only possible 
\emph{formally} around the trivial deformation
represented by $0 \in \fL^1$, and furthermore up to 
some residual equivalence. The `optimal' formal deformation
space is then a so-called miniversal family, characterised by
the property that its tangent space at the origin is
identified with $H^1(\fL)$: the true space of first-order deformations. 
The \emph{Kuranishi family} is a miniversal family realised as
an often singular formal subvariety in $H^1(\fL)$.
In our case it will turn out that the construction may
be carried out \emph{globally}, producing an actual subvariety of $H^1(\fL)$,
which will furthermore turn out to be the actual deformation space,
i.e. a \emph{universal} family.
 
The features of our deformation problem that allow for a global construction are
captured abstractly in the following notion.
\begin{defn}
A \emph{graded nilpotent DGLA} is a finite-dimensional DGLA $\fL = \bigoplus_p \fL^p$ 
together with a grading $\fL^p = \bigoplus_{i>0} \fL^p_i$ in \emph{positive} degrees
such that both the bracket and the differential are of graded degree zero.
\end{defn}

Note that we have added a `lower-case' grading to the data. In particular,
$\fL^0 = \bigoplus_{i>0} \fL^0_i$ becomes a (positively) graded nilpotent Lie algebra,
with a compatible action on the graded vector space $\fL = \bigoplus_{i>0} \fL_i$.
The next feature of our particular deformation problem is the very strong vanishing condition
$H^0(\fL)=0$, implying
that the trivial deformation has a trivial stabiliser in the gauge group. As we shall
see, this does in fact imply that the gauge group acts freely on $\MC(\fL)$.
\begin{pro}\label{pro:kura}
Let $\fL$ be a graded nilpotent DGLA with $H^0(\fL)=0$. 
Then there is an algebraic subset $M \subset H^1(\fL)$
together with an algebraic map $\pi:\MC(\fL) \to M$,
an algebraically trivial principal bundle for the gauge group.
As a consequence, the gauge action groupoid $\MC(\fL)\sslash\exp\fL^0$
is equivalent to the discrete groupoid over $M$.
\end{pro}
We will refer to such $M$, together with an
algebraic section $\xi : M \to \MC(\fL)$,
as a (global, universal) Kuranishi family.
\begin{proof}
Choose a splitting
$$\fL = Z(\fL) \oplus C = B(\fL) \oplus H(\fL) \oplus C$$
on the level of bigraded (i.e. graded Graded) vector spaces. Note that
the restriction $d|_C : C \to B(\fL)$ of the differential is invertible.
Let $\delta : \fL \to \fL$ be the map of Graded degree $-1$ (and graded degree $0$)
defined as $d|_C^{-1} : B(\fL) \to C$ pre-composed with projection $\fL \to B(\fL)$
and post-composed with inclusion $C \to \fL$. 
Note that $\delta^2=0$ and $d\delta$ is the projection onto $B(\fL)$
while $\delta d$ is the projection onto $C$.
Consider the algebraic map
$$
\Phi : \fL^1 \to \fL^1,\quad \Phi(x) = x + \frac{1}{2} \delta [x,x].
$$
We claim that it possesses the following properties:
\begin{enumerate}
\item $\Phi$ admits an algebraic inverse,
\item $\Phi$ identifies $\MC(\fL)$ with the set 
$$ \{ x \in Z^1(\fL)\ |\ 
[\Phi^{-1}x, \Phi^{-1}x]\in B^2(\fL) \},$$
\item $\Phi$ identifies 
$\MC(\fL) \cap (H^1(\fL)\oplus C^1)$
with $\Phi(\MC(\fL)) \cap H^1(\fL)$. 
\end{enumerate}
For (1), observe that the additional nilpotent grading on $\fL$
allows one to solve the equation $x + \frac{1}{2}\delta[x,x] = y$
degree by degree.
For (2), observe that given $x \in \fL^1$ such that
$[x,x] \in B^2(\fL)$, we have $d\Phi(x) = dx + \frac{1}{2}[x,x]$ indentically.
For (3), observe that given $x \in \MC(\fL)$ we have
$d\Phi(x)=0$ as well as $\delta \Phi(x) = \delta x$.
We now define $M$ to be the zero-locus of the quadratic map
$$ H^1(\fL) \hookrightarrow \fL^1 \xrightarrow{\Phi^{-1}} 
\fL^1 \xrightarrow{[,]} \fL^2 \to H^2(\fL)\oplus C^2 $$ 
so that $[\Phi^{-1}(x), \Phi^{-1}(x)]\in B^2(\fL)$ for all $x \in M$.
It follows that there is a pullback diagram
$$\begin{diagram}
\node{M} \arrow{s,l}{\xi} \arrow{e} \node{\fL^1} \arrow{s,r}{\Phi^{-1}} \\
\node{\MC(\fL) \cap (H^1(\fL)\oplus C^1)} \arrow{e} \node{\fL^1.}
\end{diagram}$$
identifying $M$
with the zero-locus of $\delta$ in $\MC(\fL)$.

We now ask whether this zero-locus intersects all gauge orbits.
Given $x \in \MC(\fL)$, define the function
$$ \Psi_x : C^0 \to C^0,\quad \Psi(y) = \delta(e^y\ast x). $$
Again using the nilpotent grading on $\fL$ and solving the
equation $\delta (e^y\ast x) = z$ degree by degree, one finds
that $\Psi_x$ admits an algebraic inverse, and furthermore
the latter depends algebraically on $x$ (the degree $i$ graded component
of this equation is of the form $y_i=\delta d y_i = \delta(\dots)$ 
where $(\dots)$ involves
only $y_j$, $j<i$). Thus,
the zero-locus of the map
$$ C^0 \times \MC(\fL) \to C^0,\quad (y,x) \mapsto \Psi_xy $$
is precisely the graph of an algebraic map $\eta : \MC(\fL) \to C^0$
such that $$e^{\eta(x)}\ast x \in H^1(\fL)\oplus C^1(\fL)$$
for all $x \in \MC(\fL)$. In particular, we may define
the projection
$$ \pi : \MC(\fL) \to M,\quad \pi(x) = \Phi(e^{\eta(x)}\ast x) $$
so that $\pi\circ\xi = \id_M$.

Of course $C^0 = \fL^0$ by the hypothesis $H^0(\fL)=0$.
It thus follows that the action map
$$
\exp\fL^0 \times M \times  \MC(\fL),\quad (e^y,m) \mapsto e^y \ast \xi(m)
$$
admits an algebraic inverse
sending $x \in \MC(\fL)$ to $(e^{-\eta(x)}, \pi(x)) \in \exp\fL^0\times M$.
Hence, $\pi : \MC(\fL) \to M$ is a trivial principal bundle
for the gauge action of $\exp\fL^0$. Finally, $\pi$ induces
a homomorphism of groupoids from $\MC(\fL)\sslash\exp\fL^0$ to $M$ (discrete),
sending the morphism $x \to x'$ given by $u \in \exp\fL^0$
to the identity morphism of $\pi(x)=\pi(x')$. By freeness
of the action of the structure group of a principal bundle,
it follows that the automorphism group of an object in $\MC(\fL)\sslash\exp\fL^0$
is trivial and thus the induced maps on automorphism groups are isomorphisms. Thus,
the above homomorphism is an equivalence of groupoids.
\end{proof}

We remark that even without the condition $H^0(\fL)=0$ we may
carry out the construction of $M \subset H^1(\fL)$ together
with the quotient map $\pi : \MC(\fL) \to M$ and a section
$\xi$ identifying $M$ with $\MC(\fL) \cap (H^1(\fL) + C^1)$.
However, if $H^0(\fL)$ is nontrivial, $\pi$ is no longer a principal bundle.
Nevertheless, $M$ may still be a universal family, i.e. the true quotient
$\MC(\fL) / \exp \fL^0$, as long as the gauge subgroups stabilising all 
Maurer-Cartan elements are of the same dimension $\dim H^0(\fL)$.

\subsection{Back to deformations: fibres}

We now return to the previous setting. Our aim is to describe
the fibre $\Def_{\u\fk}$ as equivalent to the action groupoid
$\MC(\fL)\sslash \exp\fL^0$ for a suitable DGLA controlling 
filtered deformations of $\u\fk$ with a trivial associated graded. This
will then allow us to pass to the Kuranishi
space $M$, as a discrete groupoid (as long as the zeroth cohomology of $\fL$ vanishes).
In particular, the points of $M \subset H^1(\fL)$ will
be in bijection with isomorphism classes of objects of $\Def_{\u\fk}$. We refer
to~\cite{Fialowski}
for the background on deformations of Lie algebras.

We begin by recalling some further Graded notions,
following the conventions introduced
in the preceding subsection.
\begin{defn}\ 
\begin{enumerate}
\item A Graded-commutative algebra is a Graded vector space $A = \bigoplus_p A^p$
together with an associative bilinear operation $\cdot$ of degree $0$ such that
$$ xy = (-1)^{|x||y|}yx $$
for homogeneous elements $x,y \in A$. 
\item A Graded derivation of $A$ of degree $r$
is a degree $r$ linear map $\delta : A \to A$ satisfying the graded Leibniz identity
$$
 \delta (xy) = (\delta x)y + (-1)^{} x (\delta y).
$$
\item A Graded Lie algebra is a Graded vector space $\fL = \bigoplus_p \fL^p$
together with a bracket $[,] : \fL \otimes \fL$ of degree $0$ satisfying
Graded skewness and Graded Leibniz rule (cf. Definition \ref{def:dgla}). 
\item The Graded Lie algebra of derivations of a Graded-commutative algebra $A$
is the Graded vector space $\Der A = \bigoplus_p \Der^p A$ where
$\Der^p A$ consists of degree $p$ Graded derivations of $A$, together with the
bracket defined by
$$ [ \delta, \delta' ] = \delta\delta' - (-1)^{|\delta||\delta'|} \delta'\delta $$
on homogeneous elements, and extended by bilinearity.
\end{enumerate}
\end{defn}

Let us now fix a graded Lie algebra $\u\fk$.
The exterior algebra $\Lambda^\bullet\u\fk^*$ 
is naturally a Graded-commutative algebra, and
$\u\fk$ embeds in $\Der^{-1} \Lambda^\bullet\u\fk^*$
via the evaluation map.
Consider the cochain complex 
$$C^\bullet(\u\fk,\u\fk) = \Lambda^\bullet \u\fk^* \otimes \fk$$ 
with differential $d$ computing Lie algebra
cohomology with coefficients in the adjoint representation. 
We define an injective map
$$ i : \Lambda^\bullet\u\fk^* \otimes \u\fk \to \Der \Lambda^\bullet\u\fk^* $$ 
$$ i_{\alpha\otimes X} \beta = \alpha \wedge \beta(X). $$
such that $i_{\alpha\otimes X}$ is a degree $|\alpha|-1$ Graded derivation
for homogeneous $\alpha$.
Let $\widetilde\fL(\u\fk)$ be the image of $i$. One checks that
it is a Graded Lie sub-algebra of $\Der \Lambda^\bullet\u\fk^*$,
and furthermore becomes a DGLA with the differential $d$.
Explicitly, we have 
$$ \tilde\fL^p(\u\fk) \simeq C^{p+1}(\u\fk,\u\fk). $$
The bracket in $\Der \Lambda^\bullet\u\fk^*$, restricted to
$\tilde\fL(\u\fk)$ and viewed as a bilinear operation on 
$C^\bullet(\u\fk,\u\fk)$ is usually referred to as the \emph{Richardson-Nijenhuis bracket}.

Since $\u\fk$ was originally a graded Lie algebra,
it follows that $\tilde\fL(\u\fk)$ carries additionally an induced grading
$\tilde\fL(\u\fk) = \bigoplus_i \tilde\fL(\u\fk)_i$.
While $\tilde\fL(\u\fk)$ controls all deformations of $\u\fk$,
we need to restrict to a subalgebra corresponding to filtered
deformations whose associated graded is trivial. That corresponds
to taking the components of strictly positive degree: we define
$$\fL(\u\fk) = \bigoplus_{j>0} \tilde\fL(\u\fk)_j.$$
Since $d$ and the Richardson-Nijenhuis bracket
are compatible with the grading induced from $\u\fk$,
it follows that $\fL(\u\fk)$
is a \emph{graded nilpotent sub-DGLA.}
We leave it to the reader to 
convince themselves that
the assignment $\u\fk \mapsto \fL(\u\fk)$ extends to a functor
$$\fL : \GLA \to \DGLA$$
to the category of finite-dimensional DGLAs and their homomorphisms.
Its relation to deformations may be finally revealed:
\begin{pro}\label{pro:fibres}
For each $\u\fk$ in $\Sub$, there is an equivalence of categories
between $\Def_{\u\fk}$ and $\MC(\fL(\u\fk)) \sslash \exp\fL^0(\u\fk)$.
\end{pro}

We need a pair of Lemmas.
\begin{lem}\label{lem:def-phi}
Let $\u\fk$ be a graded Lie algebra, and $\phi \in \fL^1(\u\fk)$.
Set $\fk_\phi = \u\fk$ as filtered vector spaces.
Viewing $\phi$ as a map $\Lambda^2\u\fk^* \to \u\fk$, define a
deformed bracket on $\fk_\phi$ by
$$
[ X, Y ]_\phi = [X,Y] + \phi(X,Y).
$$
Then $[-,-]_\phi$ satisfies the Jacobi identity if and only if
$\phi$ is Maurer-Cartan. If that is the case, $\fk_\phi$ with $[,]_\phi$
is a filtered Lie algebra such that the tautological map $\gr\fk_\phi\to\u\fk$ is an isomorphism
of graded Lie algebras.
\end{lem}
\begin{proof}
We compute:
\begin{eqnarray*}
[[X,Y]_\phi,Z]_\phi &=& [[X,Y],Z] + [\phi(X,Y),Z] + \phi([X,Y],Z) 
\\ &+&
2\phi(\phi(X,Y),Z) 
\end{eqnarray*}
so that Jacobi identity is equivalent to
$$ [Z, \phi(X,Y)] - \phi([Z,X],Y) + 2\phi(Z,\phi(X,Y)) + \textrm{cycl.} = 0. $$
Now, the differential $d\phi$ is precisely
$$
d\phi(X,Y,Z) = [Z,\phi(X,Y)] - \phi([Z,X],Y) + \textrm{cycl.}
$$
On the other hand, the Richardson-Nijenhuis bracket is given by
$i_{[\phi,\phi]} = 2 i_\phi^2$ whence 
$$
[\phi,\phi](X,Y,Z) = 4\phi(\phi(X,Y),Z) + \textrm{cycl.}
$$
and thus the Jacobi identity becomes $d\phi + \frac{1}{2}[\phi,\phi]=0$.
Now, if the above is satisfied, the deformed bracket is compatible
with the filtration since $\phi$ is contained in the degree $0$
filtered piece of $C^2(\u\fk,\u\fk)$;
furthermore, it induces the original bracket on the associated graded since
$\phi$ is in fact contained in the degree $1$ filtered piece.
\end{proof}

\begin{lem}
\label{lem:def-u}
Let $U \subset \GL(\u\fk)$ be the unipotent subgroup
consisting of filtration-preserving maps $u : \u\fk \to \u\fk$
with $\gr u = \id_{\u\fk}$, and let $\fu \subset \End\u\fk$
denote its Lie algebra. 
Then the embedding $C^{1,1}(\u\fk,\u\fk) \hookrightarrow \End \u\fk$
identifies $\fL^0(\u\fk)$ with $\fu$ and 
induces an identification of $\exp\fL^0(\u\fk)$ with $U$
such that, in the notation of Lemma \ref{lem:def-phi},
$$ u : \fk_\phi \to \fk_{u\ast\phi} $$
is a filtered Lie algebra isomorphism for all $u \in U$, $\phi \in \MC(\fL^1(\u\fk))$.
\end{lem}
\begin{proof}
Recall that the gauge action of $\exp\fL^0$ on $\fL^1(\u\fk)$
may be identified with the restriction of its linear action on $\fL^1(\u\fk) \oplus \langle d\rangle$
to the linear subspace $d + \fL^1(\u\fk)$. In the present case, we may embed 
$\fL(\u\fk) \oplus \langle d \rangle$ as a sub-DGLA of $\tilde\fL(\u\fk)$, where
the additional degree $1$ element $d$ is  mapped to $d\id$, with $\id \in \tilde\fL^0(\u\fk)
= \End\u\fk$ being the identity map. We then have 
$$
d\id + u \ast \phi = u( d\id + \phi)
$$
for all $u \in U$, $\phi \in \fL^1$. Now, computing
$$
(d \id)(X,Y) = [X,\id Y] - [Y, \id X] - \id[X,Y] = [X,Y] 
$$
we have that
$$
[uX,uY]_{u\ast\phi} = u[X,Y]_{\phi}
$$
as desired.
\end{proof}

\begin{proof}[Proof of Proposition \ref{pro:fibres}]
The construction of Lemma \ref{lem:def-phi} gives rise to a functor
$$F : \MC(\fL(\u\fk)) \sslash \exp\fL^0(\u\fk) \to \Def_{\u\fk}$$
sending $\phi$ to $\fk_\phi$ together with the tautological map
$\gr\fk_\phi \to \u\fk \subset \fg$. An element $u \in U\simeq\exp\fL^0(\u\fk)$, viewed
as a morphism $\phi \to \phi'$ as in Lemma \ref{lem:def-u}, 
is sent by $F$ to itself viewed as a filtered
map $\fk_\phi \to \fk_{\phi'}$. It is easy to see from this construction that
$F$ is full and faithful. Finally, to see that it is essentially surjective
note that for any $(\fk,\iota)$ in $\Def_{\u\fk}$ we may
choose an identification of $\fk$ with $\gr\fk \simeq \u\fk$ as vector spaces;
then define an element $\phi : \Lambda^2\u\fk^* \to \u\fk$ by
$\phi(X,Y) = [X,Y]_\fk - [X,Y]_{\u\fk}$. By Lemma \ref{lem:def-phi}, $\phi$ is Maurer-Cartan
and thus defines
an object of $\MC(\fL(\u\fk)) \sslash \exp\fL^0(\u\fk)$. By construction, $F(\phi)$
is isomorphic to $(\fk,\iota)$ in $\Def_{\u\fk}$.
\end{proof}

\begin{cor}\label{cor:fibres}
For each $\u\fk$ in $\Sub$ such that
$H^0(\fL(\u\fk))=0$ there is an equivalence of categories between
$\Def_{\u\fk}$ and the discrete groupoid on $M_{\u\fk}$,
where $(M_{\u\fk}, \xi_{\u\fk})$ is a Kuranishi family for $\fL(\u\fk)$.
\end{cor}
\begin{proof}
This follows from Proposition \ref{pro:fibres} and Proposition
\ref{pro:kura}.
\end{proof}

\subsection{The complete picture}

Having described the fibres of $\Def \to \Sub$ in terms of Kuranishi families,
we would like to assemble them together and recover the original category $\Def$.
Note that passing to the fibres we lose information about those morphisms
in $\Def$ which project to non-identity morphisms in $\Sub$. That is, these
are either (1) homomorphisms of filtered deformations of a single $\u\fk \subset \fg$,
whose associated graded is a nontrivial automorphism of $\u\fk$; or (2)
homomorphisms of filtered deformations of different graded subalgebras of $\fg$. 

If the map assigning to each $\u\fk$ in $\Sub$ the action groupoid 
$\MC(\fL(\u\fk)) \sslash \exp\fL^0(\u\fk)$ were (pseudo)functorial, we would
recover a category fibred over $\Sub$ through the so-called Grothendieck construction.
However, $\Def$ is \emph{not} a fibred category over $\Sub$, and
even the map $\u\fk \mapsto \MC(\fL(\u\fk))$ isn't functorial -- indeed, given
a graded Lie algebra monomorphism $\u\fk \to \u\fk'$, we only have a subset
of Maurer-Cartan element in $\fL(\u\fk')$ that project-restrict to $\fL(\u\fk)$. We thus
have to perform the construction by hand:
\begin{defn}$\bMC$ is the category whose objects are pairs
$(\u\fk, \phi)$ where $\u\fk$ is an object of $\Sub$ and
$\phi \in \MC(\fL(\u\fk))$. Its morphisms $(\u\fk,\phi) \to (\u\fk',\phi')$
are equivalence classes of triples $(g,u,u')$ where 
$$ g \in \Hom_\Sub(\u\fk,\u\fk'),\quad
u \in \exp\fL^0(\u\fk),\quad u' \in \exp\fL^0(\u\fk')$$
are such that the diagram
$$\begin{diagram}
\node{\Lambda^2 \u\fk'^* } \arrow{e,t}{u'^{-1}\phi'} \arrow{s} \node{\u\fk'} \\
\node{\Lambda^2 (g\u\fk)^* } \arrow{e,t}{\Ad_g u\phi} \node{g\u\fk} \arrow{n}
\end{diagram}$$
commutes. Two triples $(g,u,u')$ and $(h,v,v')$ are equivalent if and only if
$g=h$ and $u' \circ \Ad_g \circ u = v \circ \Ad_g \circ v'$ as maps $\u\fk \to \u\fk'$. 
\end{defn}

We have an obvious forgetful functor $\bMC \to \Sub$, and it is easy to
see that the fibre $\bMC_{\u\fk}$ is \emph{canonically isomorphic} (not just equivalent)
to $\MC(\fL(\u\fk)) \sslash \exp\fL^0(\u\fk)$, and thus equivalent to $\Def_{\u\fk}$.

\begin{pro}\label{pro:def-mc}
There is an equivalence of categories between $\Def$ and $\bMC$, compatible
with the forgetful functors to $\Sub$.
\end{pro}
\begin{proof}
Just as in the proof of Proposition \ref{pro:fibres},
the construction of Lemma \ref{lem:def-phi} gives a rise to a functor
$$\begin{diagram}
\node{\bMC}\arrow{se}\arrow[2]{e,t}{F}\node{}\node{\Def}\arrow{sw} \\
\node{} \node{\Sub}
\end{diagram}$$
sending $(\u\fk,\phi)$ to $\fk_\phi$ together with the tautological
map $\iota_\phi:\gr\fk_\phi \to \u\fk\subset\fg$. A pair $(g, u)$, viewed as
a morphism $(\u\fk,\phi) \to (\u\fk',\phi')$ is sent by $F$ to 
the composite
$$ \fk_\phi = \u\fk \xrightarrow{g} \u\fk' \xrightarrow{u} \u\fk' = \fk'_{\phi'}. $$
The restriction of $F$ to a fibre is the equivalence $\bMC_{\u\fk} \to \Def_{\u\fk}$
of Proposition \ref{pro:fibres}. This in particular shows that $F$ is essentially surjective.
Consider now the induced map
$$
\Hom_{\bMC}((\u\fk,\phi), (\u\fk',\phi')) \xrightarrow{F} \Hom_{\Def}((\fk_\phi,\iota_\phi),
(\fk'_{\phi'}, \iota_{\phi'})) 
$$
on homsets. We shall construct its inverse. 
Let $U \subset \GL(\u\fk)$, $U' \subset \GL(\u\fk')$
be the unipotent subgroups identified with $\exp\fL^0(\u\fk)$, $\exp\fL^0(\u\fk')$
as in the proof of Proposition \ref{pro:fibres}.
Recall that the homset on the right hand side consists of pairs $(\varphi,g)$
where $\varphi : \fk_{\phi} \to \fk'_{\phi'}$ is a filtered Lie algebra homomorphism
and $g \in G_0$ is an element (unique by Lemma \ref{lem:g-from-phi})
such that $\iota_{\phi'} \circ \gr\varphi = \iota_\phi \circ \Ad_g$.
Since $\iota_\phi, \iota_{\phi'}$ are the tautological maps
arising from the identification $\fk_\phi = \u\fk$, $\fk'_{\phi'}=\u\fk'$ as filtered
vector spaces, the above condition is simply $\gr\varphi = \Ad_g$. It follows
that $\varphi : \u\fk \to \u\fk'$ may be factored as
$\varphi = u' \circ \Ad_g \circ u$ for some $u \in U$, $u' \in U'$. Furthermore,
the triple $(g,u,u')$ is unique up to equivalence, and thus yields a well-defined
element of the homset on the left hand side. It is easy to check that the resulting
map is indeed the inverse of $F$. 
\end{proof}

Although $\u\fk \mapsto \bMC_{\u\fk}$ does not
give a (pseudo)functor from $\Sub$ to groupoids,
it becomes functorial once we restrict to isomorphisms.
We will only need the following property.
\begin{lem}\label{lem:lift-iso}
An isomorphism $g:\u\fk \to \u\fk'$ in $\Sub$ induces an isomorphism 
$\bMC_g : \bMC_{\u\fk} \to \bMC_{\u\fk'}$ of groupoids
such that for each object $(\u\fk,\phi)$ of $\bMC_{\u\fk}$
there is an isomorphism $(\u\fk,\phi) \to \bMC_g(\u\fk, \phi)$ in $\bMC$.
\end{lem}
\begin{proof}
Given an isomorphism $g : \u\fk \to \u\fk'$ in $\Sub$, 
we let the functor $\MC_g : \bMC_{\u\fk} \to \bMC_{\u\fk'}$ 
send $(\u\fk,\phi)$ to $(\u\fk', \phi')$ where $\phi'(gX,gY) = g\phi(X,Y)$
for all $X,Y \in \u\fk$. Its action on morphisms sends
$u \in \exp\fL^0(\u\fk)$ to $g u g^{-1} \in \exp\fL^0(\u\fk')$.
\end{proof}

Defining a category equivalent to $\bMC$ on the level of Kuranishi families
is somewhat cumbersome. Instead, we shall only extend the equivalence 
$\Def_{\u\fk} \approx M_{\u\fk}$ (discrete)
of Corollary \ref{cor:fibres} to a \emph{full} subcategory of $\Def$.
That is, we shall also have a way to represent morphisms
between objects of $\Def_{\u\fk}$ which map to a \emph{non-trivial}
automorphism of $\u\fk$ in $\Sub$. 
Let us first denote by
$G_0^{\u\fk} \subset G_0$ the stabiliser of $\u\fk$ under the adjoint action
and observe that it acts by automorphisms on $\u\fk$, and thus on
$\fL(\u\fk)$ and $\MC(\fL(\u\fk))$. Now, the action of $G_0^{\u\fk}$
is compatible with that of $\exp\fL^0(\u\fk)$, and thus it descends
to the quotient -- hence, also to $M_{\u\fk}$ for a Kuranishi family
$(M_{\u\fk}, \xi_{\u\fk})$.
\begin{defn}Let $\u\fk$ be an object of $\Sub$.
\begin{enumerate}
\item $\Def_{\u\fk}^*$ is the full subcategory of $\Def$ consisting of objects over $\u\fk$,
\item $\bMC_{\u\fk}^*$ is the full subcategory of $\bMC$ consisting of objects over $\u\fk$,
\end{enumerate}
\end{defn}

\begin{lem}\label{lem:bM}
Assume $H^0(\fL(\u\fk))=0$ and let $(M_{\u\fk},\xi_{\u\fk})$
be a Kuranishi family for $\fL(\u\fk)$.
Then there is an equivalence of categories between 
$\Def_{\u\fk}^*$ and $M_{\u\fk}\sslash G_0^{\u\fk}$.
\end{lem}
\begin{proof}
By Proposition \ref{pro:def-mc} way may replace $\Def_{\u\fk}^*$ with $\bMC_{\u\fk}^*$. 
Recalling the equivalence
$$ \xi : M_{\u\fk} \to \bMC_{\u\fk} \simeq \MC(\fL(\u\fk))\sslash \exp\fL^0(\u\fk) $$
we need to extend $\xi$ to a
commutative diagram of homomorphisms of groupoids
$$\begin{diagram}
\node{ M_{\u\fk} } \arrow{s} \arrow{e,t}{\xi} \node{ \bMC_{\u\fk} } \arrow{s} \\
\node{ M_{\u\fk} \sslash G_0^{\u\fk}} \arrow{e,b}{\xi^*} \node{ \bMC_{\u\fk}^*. }
\end{diagram}$$
Acting on objects, $\xi^*$ sends $m \in M_{\u\fk}$ to $\xi(m) \in \MC(\fL(\u\fk))$.
Acting on morphisms, $\xi^*$ sends $g : m \to m'$ to
$[(g,\id,u)] : \xi(m) \to \xi(m')$ where
$u \in \exp\fL^0(\u\fk)$ is the unique element
such that $u\ast g\xi(m) = \xi'(m')$.
 Since
$\MC_{\u\fk}$ has the same set of objects as $\MC_{\u\fk}^*$
and $\xi$ is essentially surjective, so is $\xi^*$. Furthermore, it straightforward to check
that sending $[(g, u, u')]$ to $g$ gives a well-defined
inverse map to
$$\Hom_{M_{\u\fk}\sslash G_0^{\u\fk}}(m,m') \to \Hom_{\bMC_{\u\fk}^*}(\xi(m), \xi(m')),
\quad g \mapsto [(g,\id,u)] $$
so that $\xi^*$ is full and faithful.
\end{proof}

\section{Associated parabolic geometries}

\subsection{Introduction}
The results of the previous section
essentially already suggest a classification algorithm.
Using $\pi_0(-)$ to denote the set of isomorphism classes
of objects of a groupoid, we have by Lemma \ref{lem:gs-def} Proposition \ref{pro:def-mc}
that $\pi_0(\Germ)$ forms a subset of the set of isomorphism classes of objects
of the small category $\bMC$. 
By Lemma \ref{lem:bM} we may already identify
$\pi_0(\bMC_{\u\fk}^*) \simeq M_{\u\fk} / G_0^{\u\fk}$ for each $\u\fk$ in 
$\Sub$ satisfying $H^0(\fL(\u\fk))=0$. It will turn out that all 
subalgebras relevant for our classification problem do satisfy this condition.
Now, for two objects of $\bMC$ to be isomorphic, the underlying
graded subalgebras of $\fg$ must be conjugate by an element of $G_0$.
Hence, choosing for each $G_0$-conjugacy class a representative,
we may embed the relevant subset of $\pi_0(\Germ)$ 
into the disjoint union of $M_{\u\fk}/G_0^{\u\fk}$
as $\u\fk$ runs through the chosen representatives.
The Kuranishi families and their quotients may be efficiently computed in a completely
algorithmic way. The remaining problem is to identify those points in each $M_{\u\fk} / G_0^{\u\fk}$
that correspond to a class of objects of $\Germ$. 

Recall that the embedding
$\Germ \to \GermSym \approx \bMC$ admits a left adjoint, giving rise to an endofunctor
(idempotent monad) on $\GermSym$ sending $(\sD,\fk)$ to $(\sD, \sym\sD)$. That is,
the objects of $\Germ$ are precisely those satisfying $\fk = \sym\sD$
(or just $\dim\fk=\dim\sym\sD$),
need to recognise them is to be able to compute the symmetry algebra $\sym\sD$ (or just
its dimension). That can be achieved using the methods of parabolic geometry.

\subsection{Algebraic Cartan connections}

\begin{defn}\label{def:acc}
Let $\fk$ be a Lie algebra, and $\fl \subset \fk$ a subalgebra.
An algebraic Cartan connection of type $(\fg,\fp)$ on $(\fk,\fl)$
is a linear map $\bar\omega : \fk \to \fg$ such that
\begin{enumerate}
 \item $\bar\omega(\fl) \subset\fp$, inducing
 \item $\fk/\fl \to \fg/\fp$ an isomorphism,
 \item $\bar\omega([X,Y]) = [\bar\omega X, \bar\omega Y]$ whenever $X \in \fl$.
\end{enumerate}
\end{defn}

\begin{defn}
Let $\fk$ be a filtered Lie algebra. An algebraic Cartan connection
$\bar\omega : \fk \to \fg$ on $(\fk,\fk^0)$ is \emph{regular} if it
is filtration preserving, and if $\gr\bar\omega : \gr\fk \to \fg$ is
a graded Lie algebra homomorphism.
\end{defn}

\begin{defn}
$\Cartan^a$ is the category whose objects are pairs $(\fk,\bar\omega)$
where $\fk$ is a filtered Lie algebra, and $\bar\omega : \fk \to \fg$ an \emph{injective}
regular algebraic Cartan connection of type $(\fg,\fp)$ on $(\fk,\fk^0)$. Its
morphisms from $(\fk,\bar\omega)$ to $(\fk',\bar\omega')$ are pairs
$(\varphi,p)$ where $\varphi:\fk\to\fk'$ is a filtered Lie algebra homomorphism
and $p \in P$ is such that $\Ad_p \circ \bar\omega = \bar\omega\circ\varphi$.
\end{defn}

Observe that the assignment sending $(\fk,\bar\omega)$ to $(\fk,\gr\bar\omega)$
extends to a forgetful functor
$$
\Cartan^a \to \Def.
$$
Given an algebraic Cartan connection $\bar\omega \in \fk^* \otimes \fg$
on $(\fk,\fl)$,
we define its curvature $\bar\Omega \in \Lambda^2\fk^* \otimes \fg$
by the formula
$$ \bar\Omega = d^\fk \bar\omega + \frac{1}{2}[\bar\omega\wedge\bar\omega]_\fg $$
where $d^\fk$ denotes the differential in the complex
$C^\bullet(\fk) \otimes \fg$ (i.e. with $\fk$ acting trivially on $\fg$),
and $[,]_\fg : \Lambda^2\fg \to \fg$ is the bracket in $\fg$. By the property
(3) in Definition \ref{def:acc}, $\bar\Omega : \Lambda^2\fk \to \fg$
factors through $\Lambda^2(\fk/\fl)$. Then, by property (2) we may form
the `curvature function'
$\bar\kappa \in \Lambda^2(\fg/\fp)^* \otimes \fg$
defined unambigously by
$$ \bar\Omega(X,Y) = \kappa(\bar\omega(X),\bar\omega(Y))\quad \textrm{for all}\quad X,Y \in \fk. $$
As is customary, we identify $(\fg/\fp)^*$ with $\fp_+$ as representations of $P$,
and view $\bar\kappa$ as a $2$-chain in the complex $C_\bullet(\fp_+,\fg)$ computing
Lie algebra homology of $\fp_+$ with values in $\fg$.
If $\fk$ is filtered, $\fl=\fk^0$ and $\bar\omega$ is regular, it follows
that $\bar\kappa \in C_\bullet(\fp_+,\fg)^1$. We say that $\bar\omega$ is
\emph{normal} if $\bar\kappa$ is a $2$-cycle.

\begin{lem}\label{lem:curv-inv}
$\bar\kappa \in C_2(\fp_+,\fg)$ is annihilated by the adjoint action of 
$\bar\omega(\fl) \subset \fp$.
\end{lem}
\begin{proof}
This is an algebraic version of the Bianchi identity, using the definition
of $\bar\kappa$ and $\bar\Omega$ and property (3) of Definition \ref{def:acc}.
\end{proof}

The key result connecting algebraic Cartan connections to filtered deformations
of graded subalgebras of $\fg$ is the following Proposition. Its proof will
be given in the next subsection, after a further excursion into Cartan geometries
on homogeneous spaces.
\begin{pro}\label{pro:acc}
Let $\Cartan^a_n$ denote the full subcategory of $\Cartan^a$ consisting of
objects $(\fk,\bar\omega)$ where $\bar\omega$ is normal. Then
the restriction of the forgetful functor induces an equivalence
of categories between $\Cartan^a_n$ and $\Def$.
\end{pro}

\subsection{Geometries on a homogeneous space}
Recall the notation of subsection \ref{ss:para}. We introduce a number
of notions describing invariant Cartan geometries on a fixed homogeneous
space. The qualification `of type
$(\fg,P)$' is implicitly understood.
\begin{defn} Let $K$ be a Lie group, and $L$ a closed subgroup.
\begin{enumerate}
\item An \emph{equivariant $P$-principal bundle} on $K/L$ is a right $P$-principal bundle
$\sG \to K/L$ together with a left $K$-action on $\sG$ commuting with the right
$P$-action.
\item A type Cartan connection $\omega \in \Omega^1_\sG \otimes \fg$
on an equivariant $P$-principal bundle $\sG \to K/L$
is \emph{invariant} if $L_k^*\omega=\omega$ for all $k\in K$.
\item An \emph{invariant Cartan geometry} 
on $K/L$ is an equivariant $P$-principal
bundle $\sG \to K/L$ together with an invariant Cartan connection $\omega \in \Omega^1_\sG\otimes
\fg$. 
\end{enumerate}
\end{defn}

Given a pair of Cartan geometries $(\sG,\omega)$ and $(\sG',\omega')$
of type $(\fg,P)$ on $M$, a gauge transformation between the two is a bundle morphism
$f : \sG \to \sG'$ such that $f^*\omega'=\omega$. Note that for $M$ connected,
$f$ is uniquely determined by the image of a single point of $\sG$. In particular,
the only gauge transformation from $(\sG,\omega)$ to itself is the identity.

Given Cartan geometries $(\sG\to M, \omega)$,
$(\sG'\to M', \omega')$ of type $(\fg,P)$ and 
a local diffeomorphism $h : M \to M'$, one forms 
the pullback Cartan geometry $(h^*\sG, \tilde h^*\omega)$
where $\tilde h : h^*\sG'\to\sG'$ is the natural projection,
and considers gauge transformations $f : \sG \to h^*\sG'$. These are,
equivalently, local diffeomorphisms $\bar f :\sG \to \sG'$ lifting $h$ and such that
$\bar f^*\omega'=\omega$. Furthermore, given
just a germ $h$ at $m \in M$ of a local diffeomorphism $M \to M'$, 
it makes sense to consider germs of a local gauge transformations
from $(\sG,\omega)$ to $h^*(\sG',\omega')$. These may be identified
with equivalence classes of pairs $(U, f_U)$ where $U \subset M$
is an open neighbourhood of $m$, and $f_U : \sG|_U \to \sG'$
is a local diffeomorphism such that $f_U^*\omega'=\omega$ and
$f|_U$ descends to a local diffeomorphism $U\to M'$ representing $h$.

\begin{defn}
$\Cartan^g$ is the category whose objects are $(K,L,\sG,\omega)$ where 
$K$ is a Lie group, $L$ a connected closed subgroup 
and $(\sG,\omega)$
an invariant regular Cartan geometry on $K/L$ such that
$L$ acts freely in the fibre of $\sG$ over the origin.
Its morphisms
from $(K,L,\sG,\omega)$ to $(K',L',\sG',\omega)$ are pairs $(\varphi,f)$
where $\varphi:\fk \to \fk'$ is a Lie algebra monomorphism,
$\varphi(\fl) \subset \fl'$ and $f$ is a germ at the origin of a local gauge transformation 
from $(\sG,\omega)$ to $h^*_\varphi(\sG',
\omega')$,
with $h_\varphi$ being the germ at the origin of a local diffeomorphism $K/L \to K'/L'$ 
induced by $\varphi$.
\end{defn}

\begin{lem}\label{lem:baromega}
Let $(K,L,\sG,\omega)$ be an object of $\Cartan^g$ together
with a point $e \in \sG$ over the origin. Then the orbit map
$$\lambda_e : K \to \sG,\quad k\mapsto ke $$
induces a Lie group monomorphism $L \to P$, and $\lambda_e^*\omega$
viewed as a left-invariant $\fg$-valued one-form on $K$
is an algebraic Cartan connection on $\fk$.
\end{lem}
\begin{proof}
The homomorphism $L \to P$ maps $\ell \in L$ to $p \in P$ such that
$\ell e = e p^{-1}$. It is injective by freeness of the action of $L$
on the fibre $eP$. Furthermore, letting $\bar\omega = \lambda_e^*\omega$
we have (1) $\bar\omega(\fl) \subset \fp$ since $L$ preserves the origin;
(2) $\fk/\fl \to \fg/\fp$ is an isomorphism since $\omega$ is a Cartan connection
over $K/L$; (3) $[\bar\omega X,\bar\omega Y] - \bar\omega([X,Y]) = 0$ for $X \in \fl$
by horizontality of the Cartan curvaure of $\omega$.
\end{proof}

\begin{lem}\label{lem:kk0}
Let $(\fk,\bar\omega)$ be an object of $\Cartan^a$.
Then there exists a Lie group $K$ with Lie algebra $\fk$,
and a Lie group monomorphism $K \supset K^0 \to P$ lifting $\bar\omega : \fk^0 \to \fp$.
\end{lem}
\begin{proof}
By the universal property of Tanaka prolongation, the map
$\fg \to \Hom(\fg_-, \fg)$ is injective, and so is thus the
adjoint representation $\gr\fk \to \End\gr\fk$. Since
the adjoint representation of $\gr\fk$ is the associated graded
map of the adjoint representation of $\fk$, it follows that
the latter is faithful. We may then let $K$ be the connected
subgroup of $\GL(\fk)$ with Lie algebra $\fk \subset \End\fk$.
Now, let $P_\fk \subset P$ be the stabiliser of
$\bar\omega(\fk) \subset \fg$ for the adjoint action. Since
$\gr\bar\omega(\gr\fk) \supset \fg_-$, it follows that the
natural Lie group homomorphism $P_\fk \to \GL(\fk)$ induced
by $\bar\omega$ is injective. In particular, we have 
a Lie algebra monomorphism $\fp_{\fk} \to \End\fk$
whose image contains $\fk^0$; the resulting homomorphism
$\fk^0 \to \fp$ coincides with the restriction of $\bar\omega$.
Since $K^0$ is connected, it follows that $K^0 \subset P_\fk$
as subgroups of $\GL(\fk)$, giving rise to the desired Lie group
homomorphism $K^0 \to P$ lifting $\bar\omega$. 
\end{proof}

The essence of the following result is contained in\cite[Prop. 1.5.15]{CS}.
It has an important interpretation in terms of the practical implementation
of homogeneous parabolic geometries in {\tt DifferentialGeometry}; I am
immensely indebted to Ian Anderson for our discussions on these matters.

\begin{lem}\label{lem:acc}
There is an equivalence of categories between $\Cartan^a$ and $\Cartan^g$.
\end{lem}
\begin{proof}
To construct a functor $F : \Cartan^g \to \Cartan^a$, let us
fix for each object $(K,L,\sG,\omega)$ a point $e \in \sG$ in the fibre
over the origin. By Lemma \ref{lem:baromega}
we then have an algebraic Cartan connection $\lambda_e^*\omega$ on $\fk$.
Defining a filtration $\fk^\bullet$ on $\fk$ by $\fk^i = (\lambda_e^*\omega)^{-1}(\fg^i)$
we set $F(K,L,\sG,\omega) = (\fk^\bullet, \lambda_e^*\omega)$. 
Given a morphismm
$(\varphi,f) : (K,L,\sG,\omega) \to (K',L',\sG',\omega')$ we let $F(\varphi,f) = (\varphi,p)$ with 
$p \in P$ such that $f(e) = e' \cdot p$ (by abuse of notation
we identify $e'$ with the unique element
in the fibre of $h^*_\varphi\sG$ over the origin mapping to $e' \in \sG'$).
Letting $\tilde h_\varphi$ be the germ at $e_K$ of a local embedding $K \to K'$
induced by $\varphi$, we then have
$$R_p^{-1} \circ f \circ \lambda_e = \lambda_{e'} \circ \tilde h_\varphi$$
as germs at $e_K$ of maps $K \to \sG'$, so that
$$\Ad_p \circ \lambda_e^*\omega = \lambda_{e'}^*\omega' \circ \varphi $$
and we have indeed defined a morphism in $\Cartan^a$.
Finally, compatibility with composition is a bit tedious, but not difficult to check.

To construct a functor $G : \Cartan^a \to \Cartan^g$, consider
for each object $(\fk,\bar\omega)$ the Lie groups $K \supset K^0$ 
as in Lemma \ref{lem:kk0}.
Set $G(\fk,\bar\omega) = (K, K^0, K\times^{K^0}P, \omega)$ where
$\omega$ is the unique invariant Cartan connection such that
$\lambda_{(e_K,e_P)}^*\omega=\bar\omega$. Furthermore, given a morphism
$(\varphi,p) : (\fk,\bar\omega) \to (\fk',\bar\omega')$ let
$G(\varphi,p) = (\varphi, R_p)$ where 
we identify $h_\varphi^*(K'\times^{K'^0}P)$ with the germ at the origin
of $K\times^{K^0}P$ via $\tilde h_\varphi$. Again, it is straightforward
to check that we have defined a functor.

There is a natural isomorphism $\id \to FG$ of endofunctors on
$\Cartan^a$, whose
component at $(\fk,\bar\omega)$ is $(\id_\fk,p)$ where $p\in P$ is
such that $(e_K,p) \in K\times^{K^0}P$ is the point used in the
construction of $F$. In the opposite direction, we have
the natural isomorphism $\id \to GF$ of endofunctors on
$\Cartan^g$, whose
component at $(K,L,\sG,\omega)$ is $(\id_\fk,f)$ where $f$ is a germ
of a gauge transformation sending the point $e \in\sG$ used
in the construction of $F$ to $(e_K,e_P) \in K\times^LP$ 
(we identify a neighbourhood of
identity in $K$ with a neighbourhood of identity in the
group used in the construction of $G$, and embed $L$ in $P$
as in Lemma \ref{lem:baromega}).
\end{proof}

\begin{lem}\label{lem:curvs}
Let $(K,L,\sG,\omega)$ be an object of $\Cartan^g$ and
$(\fk,\bar\omega)$ an object of $\Cartan^a$ such that
the two become isomorphic under the equivalence of categories
$\Cartan^g \approx \Cartan^a$ of Lemma \ref{lem:acc}. Then
the curvature function
$\kappa : \sG \to C_2(\fp_+,\fg)$ factors through the $P$-orbit
of $\bar\kappa$.
\end{lem}
\begin{proof}
This follows from the identity 
$\lambda^*_e d\omega = d^\fk \lambda^*_e\omega$
where we view $\lambda^*_e$ as pulling back
$K$-invariant $\fg$-valued one-forms on $\sG$ to elements of
$\Lambda^\bullet \fk^* \otimes \fg$. Indeed, letting
$\bar\omega = \lambda^*_e\omega$ we then have $\bar\Omega = \lambda^*_e\Omega$
where $\Omega = d\omega + \frac{1}{2}[\omega\wedge\omega]$ and 
$\bar\kappa = \lambda^*_e \kappa$ where $\kappa$ is the $K$-invariant curvature function
on $\sG$. In general it follows that $\kappa$ factors through the $P$-orbit of $\bar\kappa$,
and furthermore the latter statement is stable under isomorphisms in $\Cartan^a$.
\end{proof}

By Lemma \ref{lem:curvs}, the equivalence of Lemma \ref{lem:acc}
restricts to an equivalence between $\Cartan^a_n$
and $\Cartan^g_n$, the full subcategory of $\Cartan^g$ consisting
of those $(K,L,\sG,\omega)$ for which $\omega$ is normal.

\begin{proof}[Proof of Proposition \ref{pro:acc}]
We need to construct an essential inverse
to the obvious forgetful functor $\Cartan^a_n \to \Def$.
This is equivalent to providing
an essential inverse to the functor $A : \Cartan^g_n \to \Def$,
a composite of the above with $F : \Cartan^g_n \to \Cartan^a_n$
as in the proof of Lemma \ref{lem:acc}. Recall that we
choose for each $(K,L,\sG,\omega)$ a point $e \in \sG$ over the origin;
then $A(K,L,\sG,\omega) = (\fk, \gr\lambda_e^*\omega)$ where
the filtration on $\fk$ is induced by $\lambda_e^*\omega$. On morphisms
we have $A(\varphi,f) = (\varphi, g)$ where $g$ is the class in $G_0$
of the element $p \in P$ such that $f(e)=e'p$. 

To define a functor
in the opposite direction, choose for each $(\fk,\iota)$ in $\Def$
Lie groups $K^0 \subset K$ with Lie algebras
$\fk^0\subset\fk$ such that $K$ acts faithfully on $K/K^0$, and
let $\sD_\fk$ be the invariant distribution on $K$ corresponding to $\fk^{-1}/\fk^0$.
Use Proposition \ref{pro:equiv-para} to construct a regular normal
Cartan geometry $(\sG_\fk, \omega_\fk)$ over $K/L$ for the underlying datum of $\sD_\fk$.
Set $B(\fk,\iota)=(K,K^0,\sG_\fk,\omega_\fk)$.
Again by Proposition \ref{pro:equiv-para}, 
given a morphism $(\varphi,g) : (\fk,\iota) \to (\fk',\iota')$ in $\Def$,
there is a unique germ of a gauge transformation $f : (\sG_\fk,\omega_\fk) 
\to h_\varphi^*(\sG_{\fk'},\omega_{\fk'})$.
Set $B(\varphi,g)=(\varphi,f)$. We have thus defined a functor $B : \Def \to \Cartan^g_n$.

Now, $AB \simeq \id$ as endofunctors on $\Cartan^g_n$ by Proposition \ref{pro:equiv-para}
(uniqueness of a regular normal parabolic geometry up to unique gauge). On the other hand
$BA \simeq \id$ as endofunctors on $\Def$ since any two graded Lie algebra monomorphisms
$\gr\fk \rightrightarrows \fg$ whose image contains $\fg_-$ differ by a unique element of $G_0$.
\end{proof}

\subsection{Invariants and symmetries}

Observe that Proposition \ref{pro:acc} implies
a slightly stronger statement: for every object $(\fk, \iota)$ of $\Def$, we may
lift $\iota : \gr\fk \to \fg$ to a regular normal algebraic Cartan connection
$\bar\omega : \fk \to \fg$, and furthermore that latter lift is unique up to
the adjoint action of $P_+$ on $\fg$. Now, by Lemma \ref{lem:curvs} the
curvature function $\bar\kappa \in Z_2(\fp_+,\fg)$
coincides with the curvature function of
the parabolic geometry associated with the germ of a distribution described
by $\fk$, and in particular its homology class $\bar\kappa_H \in H_2(\fp_+,\fg)$
\emph{is} the harmonic curvature. More precisely, given $(\fk,\iota)$
the curvature $\bar\kappa$ is defined uniquely up to $P_+$-conjugacy,
and the harmonic curvature $\bar\kappa_H$ is an honest invariant.
\begin{lem}\label{lem:curv-fun}
The assignment sending an object $(\fk,\iota)$ of $\Def$
to the harmonic curvature of a regular normal algebraic Cartan connection
lifting $\iota$ extends to a functor
$ \Def \to H_2(\fp_+,\fg)^1 \sslash G_0$.
\end{lem}
\begin{proof}
The functor in question sends 
a morphism $(\varphi, g) : (\fk,\iota) \to (\fk,\iota')$
to $g \in G_0$.
\end{proof}

We will thus be able to compute the basic invariant $\bar\kappa_H$
of $(\fk,\iota)$ by computing a regular, normal algebraic Cartan
connection lifting $\iota$, and then its harmonic curvature. We shall
not go into the technical details of this calculation in this paper.
As we have stated in the introductory subsection, the main reason for
using the parabolic geometric description is an explicit algorithm computing 
the infinitesimal symmetries. Recall from \cite[Lemma 1.5.12]{CS} that the
infinitesimal symmetries of a Cartan geometry $(\sG \to M, \omega)$
of type $(\fg,P)$, viewed as vector fields on $\sG$ whose flow preserves
$\omega$, are in one-to-one correspondence with sections of the associated
bundle $\sG \times^P \fg$ parallel for the \emph{modified tractor connection}. 
Identifying $\Gamma(M, \sG\times^P\fg)$ with $C^\infty(\sG,\fg)^P$, we
say that a $P$-equivariant function $x : \sG \to \fg$ is parallel if
$$ dx + \omega \cdot x + \tilde\kappa\cdot x = 0 $$
where $\tilde\kappa \in \Omega^1_\sG \otimes \End\fg$ is given by
$\tilde\kappa(\xi)\cdot X =\kappa(\omega(\xi),X)$ for $\xi \in T\sG$ and $X \in \fg$.
Translating this prescription to the algebraic setting, we have the following.
\begin{lem}\label{lem:sym}
Let $(\fk,\bar\omega)$ be an object of $\Cartan^a_n$.
Define
$$ \alpha : \fk \to \End\fg,\quad \alpha(X)Y = [\bar\omega X,Y] + \bar\kappa(\bar\omega X,
Y)\quad\textrm{for all}\ X\in\fk, Y\in\fg $$
and let
$$ R = d^\fk \alpha + \frac{1}{2}[\alpha\wedge\alpha]_{\End\fg} $$
where $d^\fk$ is the differential in $C^\bullet(\u\fk) \otimes \End\fg$
(trivial action on coefficients) and $[,]_{\End\fg}$ the bracket in $\End\fg$.
Let $\fs \subset \fg$ be the subspace annihilated by endomorphisms of the form
$$ \alpha(X_1)\cdots \alpha(X_{r-2}) R(X_{r-1},X_r),\qquad r\ge 0,\quad X_1,\dots, X_r \in \fk.
$$
Then $\dim\fs= \dim \sym\sD$
where $\sD$ is a germ of a Monge distribution corresponding to $(\fk,\bar\omega)$
under the composite forgetful functor $\Cartan^a_n \to \Germ$. 
\end{lem}

For the sake of brevity, the following is a sketch of an argument.
I owe the idea of using formal holonomy to Ian Anderson. 
\begin{proof}
Consider a regular normal invariant Cartan geometry
$(K,K^0,\sG,\omega)$ corresponding to $(\fk,\bar\omega)$ under the equivalence
$\Cartan^a_n \approx \Cartan^g_n$, together with a point $e \in \sG$
over the origin such that $\bar\omega = \lambda_e^*\omega$. Then
$\alpha$ and $R$ correspond precisely to the `modified tractor connection' on
$\sG\times^P\fg$ and its curvature, pulled back by $\lambda_e$. 
Since every parallel section,
represented by a map $x:\sG \to \fg$, is determined
by its value $x(e) \in \fg$, it is enough to find the image in $\fg$
of the space of germs of parallel sections under the evaluation map at $e$. 
Since the latter map is an injection into a finite-dimensional space,
we may in fact consider infinite
jets, rather than germs, of parallel sections. Then the desired
subspace of $\fg$ is the kernel
of the `formal holonomy group', i.e. the subspace $\fs \subset \fg$
defined in the Lemma. 
\end{proof}

\section{Classification algorithm}

\subsection{Introduction}
We have established a sufficient foundation for setting up
our classification algorithm. Recall that we have passed through
the following chain of categories
$$
 \Germ \rightleftarrows \GermSym \approx \Def \approx 
\Cartan^a_n,\qquad
\Def^*_{\u\fk} \approx M_{\u\fk}\sslash G_0^{\u\fk} 
$$
where the latter holds if $H^0(\fL(\u\fk))=0$.
In addition, given an object $(\fk,\bar\omega)$ of $\Cartan^a_n$ we
may compute the harmonic curvature $\bar\kappa_H$
and the dimension of the symmetry algebra of the
corresponding distribution in terms of $\bar\omega$ (see Lemmas \ref{lem:curv-fun}
and \ref{lem:sym}).

The algorithm will pass through
equivalence classes of graded subalgebras in $\fg$.
Since we are interested
in \emph{non-flat} $2$-transitive models with vanishing scalar
component of harmonic curvature, we may from the outset restrict
to conjugacy classes represented by $\u\fk$ such that $\dim\u\fk_0 \ge 2$
and $\u\fk_0$ annihilates a non-zero element of the quintic component
of $H_2(\fp_+,\fg)$. Indeed, by Lemma \ref{lem:curv-inv}, the harmonic
curvature $\kappa_H$ of any $(\fk,\iota)$ in $\Def_{\u\fk}$
is invariant under $\bar\omega(\fk^0) \subset \fp$,
and thus under $\gr\bar\omega(\fk_0) = \u\fk_0 \subset \fg_0$, where
$\bar\omega$ is some regular, normal algebraic Cartan connection
lifting $\iota$. These graded subalgebras turn out to be
very simple:
\begin{lem}\label{lem:prol}
Suppose $\u\fk \subset \fg$ is a graded subalgebra
containing $\fg_-$ and such that $\u\fk_0$ annihilates a nonzero element of the quintic
component of $H_2(\fp_+,\fg)^1$. Then $\u\fk_i = 0$ for all $i>0$.
\end{lem}
\begin{proof}
Let $W \subset H_2(\fp_+,\fg)^1$ denote the quintic component;
it decomposes into one-dimensional weight subspaces with
weights $\lambda_j=\alpha_2 + (j-2)\alpha_1$, $0 \le j \le 5$. We want
to show, for each $i>0$, that the map
$$ \fg_i \otimes W \to \fg_i \otimes W, \quad
X \otimes \upsilon \mapsto [X,-]\cdot\upsilon $$
is an isomorphism (where we identify $\fg_i \simeq \fg_{-i}^*$
via the Killing form). Now, the above is $G_0$-equivariant
and one easily checks that the multiplicity of each
irreducible component of $\fg_i\otimes W$ is $1$, whence
the map is a scalar on every irreducible component. In particular,
it is enough to check that $[X_\alpha, -]\cdot \upsilon_j \neq0$
for every weight vector $\upsilon_j \in W_{\lambda_j}$ and
root vector $X_\alpha \in \fg_\alpha$, $\Ht \alpha=i$. 
Using the Cartan matrix of $\fg$ one computes
$$ 
[X_\alpha, X_{-\alpha}]\cdot \upsilon_j = \langle \lambda_j,\alpha^\vee\rangle \upsilon_j\neq 0
$$
whenever either $j \in \{0,5\}$ or $\alpha=\alpha_3$. In the remaining cases,
we  have $$
[X_{\alpha}, X_{-\alpha \pm \alpha_1}] \cdot \upsilon_j 
\sim X_{\pm\alpha_1} \cdot \upsilon_j
= \upsilon_{j\pm 1} \neq 0 $$
whenever $-\alpha\pm\alpha_1$ is a root (which holds at least for one choice of sign).
\end{proof}

In particular, we find that all graded subalgebras $\u\fk$
we need to consider do satisfy the cohomology vanishing property, and
thus their filtered deformations admit universal global Kuranishi families.
\begin{cor}\label{cor:prol}
Suppose $\u\fk \subset \fg$ is a graded subalgebra
containing $\fg_-$ and such that $\u\fk_0$ annihilates a nonzero element of the quintic
component of $H_2(\fp_+,\fg)^1$. Then $H^{1,1}(\u\fk,\u\fk)=H^0(\fL(\u\fk))=0$.
\end{cor}
\begin{proof}
First, by Lemma \ref{lem:prol}, we have $\u\fk = \fg_- \oplus \u\fk_0$.
Now, recall that $H^{1,1}(\u\fk,\u\fk)$ is the space of 
positive degree derivations of $\u\fk$. Note that
since $\u\fk_0$ embeds into $\End \fg_{-1}$,
every $\delta \in H^1_i(\u\fk,\u\fk)$, $i>0$
is determined by the induced map
$\bigotimes^{i+1} \fg_{-1} \to \fg_{-1}$ sending $X_0\otimes\cdots\otimes X_i$
to $[X_i,[\cdots[X_1,\delta X_0]]]$.
Thus, by the universal property of Tanaka prolongation,
we have an embedding
$\u\fk \oplus H^{1,1}(\u\fk,\u\fk) \to \fg$
as a graded Lie subalgebra.
Using Lemma \ref{lem:prol}
again, we find that the latter subalgebra is contained in non-positive degrees,
whence $H^{1,1}(\u\fk,\u\fk)=0$.
\end{proof}

The representing
map $\xi_{\u\fk} : M_{\u\fk} \to \MC \fL(\u\fk)$ 
defines for each point $m \in M_{\u\fk}$ a filtered deformation
$\fk_m$ of $\u\fk$ together with a monomorphism $\iota_m : \gr\fk_m \to \fg$, 
and we may furthermore choose a lift of $\iota_m$ to a regular, normal
algebraic Cartan connection $\bar\omega_m : \fk_m \to \fg$. This allows
us to compute the harmonic curvature $\bar\kappa_m$ and the symmetry
dimension $d(m)$ computed as in Lemma \ref{lem:sym}. Points $m$
at which $d(m) > \dim\u\fk$ are rejected.
In practice, these computations are performed globally over $M_{\u\fk}$:
viewing the latter as an algebraic subvariety in $H^{2,1}(\u\fk,\u\fk)$,
the family $\{\fk_m\}$ and the data of $\bar\omega_m$ are seen as
a filtered Lie algebra and an algebraic Cartan connection over the
ring $\RR[M_{\u\fk}]$. 

It is a non-trivial observation that the method
computing a regular normal algebraic Cartan connection may be
carried out over a ring and gives rise to a well-defined connection
upon reduction to the residue field at each point. The situation
with the symmetry dimension $d(m)$ is different, as $d : M_{\u\fk} \to \ZZ$
is Zariski upper-semicontinuous. In practice, one computes the
symmetry dimension at the generic point of each irreducible component of $M_{\u\fk}\otimes_\RR\CC$,
applying the prescription of Lemma \ref{lem:sym} in terms of linear algebra over
the field of rational functions.
If it is strictly greater than $\dim\u\fk$, the
component may be rejected (by upper semi-continuity). Otherwise, one would proceed
to find a proper Zariski-closed subset at which $d$ jumps, decompose it into irreducible
components, compute $d$ at generic points etc. In the end, we shall not go that far
in the present article: the families we describe might contain finitely
many special points corresponding to models that have already been included in 
another family. 
Here by $M_{\u\fk}\otimes_\RR\CC$
we mean the subvariety of $\CC \otimes_\RR H^1(\fL(\u\fk))$
cut out by the quadratic condition
$[\Phi^{-1}x,\Phi^{-1}x]\in B^2(\fL(\u\fk))$ as in the proof of
Proposition \ref{pro:kura}; in particular, the algebraic set $M_{\u\fk} \subset H^1(\fL(\u\fk))$
is its set of real points. 
Interestingly, the irreducible components
of $M_{\u\fk} \otimes_\RR \CC$ turn out to be defined over $\RR$, and in fact rational.

\subsection{Ingredients and recipe}
We review the computational tasks forming the basic ingredients of our method:
\begin{enumerate}
\item to parameterise $G_0$-conjugacy classes of subalgebras $\u\fk_0 \subset \fg_0$
annihilating an element in the quintic component of $H_2(\fp_+,\fg)^1$;
\item given $\u\fk = \fg_- \oplus \u\fk_0$, to compute $H^{2,1}(\u\fk,\u\fk)$, a
Kuranishi family $(M_{\u\fk}, \xi_{\u\fk})$ and the quotient $M_{\u\fk} / G_0^{\u\fk}$.
\item given $(\fk,\iota)$, to compute a regular, normal algebraic Cartan connection
$\bar\omega : \fk \to \fg$ lifting $\iota$;
\item given $(\fk,\bar\omega)$, to compute the harmonic curvature $\bar\kappa_m$ and
symmetry dimension $d(m)$.
\end{enumerate}
As we have remarked, whenever constructions are performed in families,
we use linear algebra over a suitable ring. It will become clear later on
that in case of a continuous family of conjugacy classes, say parameterised
by a variable $\lambda$, one will first
need to pass to the field $\RR(\lambda)$ of rational functions,
consider $\u\fk$ as a Lie algebra over $\RR(\lambda)$,
then perform the construction of a Kuranishi family $M_{\u\fk}$
\emph{over} $\RR(\lambda)$, and further work in the ring of functions
$\RR(\lambda)[M_{\u\fk}]$ when computing algebraic Cartan connections.
We do not devote any more space to the exploration of these issues.

Let us recall a standard trick facilitating the computation of Lie algebra cohomology:
given a diagonalizable subalgebra $\fa\subset\u\fk$, we may restrict all considerations to the
sub-complex of $C^\bullet(\u\fk,\u\fk)$ annihilated by $\fa$: its $\fa$-invariant
complement is homotopic
to zero (this had been pointed out to me by Ian Anderson).
Now, the Nijenhuis-Richardson bracket restricts to this sub-complex, and we thus have
a sub-DGLA $\fL(\u\fk)^\fa$ quasi-isomorphic to $\fL(\u\fk)$. 
In particular, $H^1(\fL(\u\fk)) = H^1(\fL(\u\fk)^\fa)$ and
a Kuranishi family constructed for
$\fL(\u\fk)^\fa$ provides a Kuranishi family for $\fL(\u\fk)$ itself.
We 
also note that if $H^2_i(\fL(\u\fk))=0$ for all $i\ge 2$,
then $M_{\u\fk}$ is the entire $H^1(\fL(\u\fk))$ (there are
no obstructions). In terms of Lie algebra cohomology,
we have $$H^{3,2}(\u\fk,\u\fk)=0 \implies M_{\u\fk} = H^{2,1}(\u\fk,\u\fk).$$

\begin{thm}
The following algorithm yields a complete list, without repetitions,
of all non-flat, at least $2$-transitive homogeneous models of $C_3$ Monge 
geometries with vanishing scalar component of the harmonic curvature.
First find the set of
$G_0$-conjugacy classes of graded subalgebras $\fg_- \subset \u\fk\subset \fg$
such that $\u\fk_0$ is at least two-dimensional, and
annihilates a nonzero element of the quintic component of $H_2(\fp_+,\fg)^1$.
Then, for each class perform the following sequence:
\begin{enumerate}
\item Fix a representative $\u\fk$ and 
compute a Kuranishi family $(M_{\u\fk}, \xi_{\u\fk})$.
\item Compute the family $(\fk_m)_{m\in M_{\u\fk}}$
of filtered deformations of $\u\fk$ defined by $\xi_{\u\fk}$.
\item Find a family of regular, normal Cartan connections
$\bar\omega_m : \fk_m \to \fg$, $m \in M_{\u\fk}$.
\item 
Compute the algebraic map
$$\kappa_H : M_{\u\fk} \to H_2(\fp_+,\fg)^1$$
such that $\kappa_H(m)$ is the harmonic curvature of $\omega_m$.
\item
Compute the Zariski upper-semicontinuous map
$$ d : M_{\u\fk} \to \ZZ $$
such that $d(m)$ is the symmetry dimension of $\omega_m$ (Lemma \ref{lem:sym}).
\item
Let $M_{\u\fk}' \subset M_{\u\fk}$ be the algebraic subset cut out
by the scalar component of $\kappa_H$, and $M_{\u\fk}'' \subset M_{\u\fk}'$
its (possibly empty) intersection with the Zariski-open subset on which $d = \dim\u\fk$.
\item
Compute the set-theoretic quotient $M_{\u\fk}''/G_0^{\u\fk}$
and append its points to the list of homogeneous models.
\end{enumerate}
\end{thm}

\begin{proof}
Our goal is to describe the subset of $\pi_0(\Germ)$
corresponding to non-flat, at least $2$-transitive models
with vanishing scalar component of the harmonic curvature.
According to Lemmas \ref{lem:germsym} and \ref{lem:gs-def}
this is equivalent to describing the set of isomorphism classes
of objects of the sub-groupoid of $\Def$ consisting of objects $(\fk,\iota)$
satisfying the following properties:
\begin{enumerate}
\item the image of $(\fk,\iota)$ in $\Sub$
is of the form $\u\fk = \fg_- \oplus \u\fk_0$
where $\u\fk_0$ is at least $2$-dimensional and annihilates a non-zero
element of the quintic component of $H_2(\fp_+,\fg)^1$,
\item the image of $(\fk,\iota)$ in $H_2(\fp_+,\fg)^1 \sslash G_0$
is a non-zero element of the quintic component of $H_2(\fp_+,\fg)^1$,
\item the symmetry dimension of 
$(\fk,\iota)$ equals the dimension of $\u\fk$.
\end{enumerate} 
By Proposition \ref{pro:def-mc} and Lemma \ref{lem:lift-iso}, 
it is enough to restrict to a set of representatives
of $G_0$-conjugacy classes of graded subalgebras $\u\fk = \fg_- \oplus \u\fk_0$,
and thus to a set of representatives of conjugacy classes
of subalgebras $\u\fk_0 \subset \fg_0$. 

More explicitly, denote by $\Sigma$
a set consisting of precisely one representative for each conjugacy
class of subalgebras $\u\fk_0 \subset \fg_0$ such that $\dim\u\fk_0 \ge 2$
and $\u\fk_0$ annihilates a non-zero element in the quintic component of
$H_2(\fp_+,\fg)^1$. 
Then the desired subset of $\pi_0(\Germ)$
may be identified with the set of isomorphism classes
of objects $(\fk,\iota)$ of $\Def$ such that:
\begin{enumerate}
\item $\iota(\fk)_0 \in \Sigma$,
\item (as above)
\item the symmetry dimension of $\fk$ equals $\dim\fk$.
\end{enumerate}
Furthermore, since no two subalgebras in $\Sigma$
are conjugate, any isomorphism between two such objects
induces an isomorphism of their images in $\Sub$, and thus
is necessarily a morphism of $\Def^*_{\fg_-\oplus\u\fk_0}$ for some
$\u\fk_0 \in \Sigma$.
Now, for each $\u\fk_0 \in \Sigma$,
letting $\u\fk = \fg_- \oplus \u\fk_0$, we have
$\Def^*_{\u\fk} \approx M_{\u\fk} \sslash G_0^{\u\fk}$
with a Kuranishi family $(M_{\u\fk},\xi_{\u\fk})$ (Corollary \ref{cor:prol} and
Lemma \ref{lem:bM}). Hence the desired subset of $\pi_0(\Germ)$
may be identified with the subset
$$ \sM \subset \coprod_{\u\fk_0 \in \Sigma} M_{\fg_- \oplus \u\fk_0} / G_0^{\fg_- \oplus \u\fk_0} $$
consisting of points $m \in M_{\u\fk}$ such that:
\begin{enumerate}
\item $\bar\kappa_{H,m}$ is a non-zero element of the quintic component of $H_2(\fp_+,\fg)^1$,
\item the symmetry dimension of $(\fk_m, \bar\omega_m)$ computed by Lemma \ref{lem:sym}
equals $\dim\u\fk$, 
\end{enumerate} 
where $\fk_m$ is the deformed filtered Lie algebra structure on $\u\fk$ defined
by $\xi_{\u\fk}(m)$, and $\bar\omega_m : \fk_m \to \fg$ is a regular, normal algebraic
Cartan connection with harmonic curvature $\bar\kappa_{H,m}$.
A careful examination of the steps (1)-(7) of the algorithm shows that
its outcome is precisely $\sM$ above.
\end{proof}

One may view the above Theorem as reducing the classification problem
of homogeneous models to the classification problem of subalgebras
$\u\fk_0 \subset \fg_0$ with suitable properties. The success of this approach
relies thus on being able to solve the latter task. In general, it is completely
hopeless; here however, since our $\fg_0$ has semisimple rank $1$, everything is
as simple as the representation theory of $\fsl(2,\RR)$.

\section{Application}

\subsection{Subalgebras}
We are ready to begin implementing the above algorithm. Explicit
calculations (cohomology, connections, symmetries) are
performed using Ian Anderson's {\tt DifferentialGeometry} package
for {\sc Maple}.
Recall that $\fg_0 \simeq \fsl(2,\RR) \oplus
\RR^2$. Let $H,X,Y$ be a standard basis in the $\fsl(2,\RR)$ factor,
with $[H,X]=2X$, $[H,Y]=-2Y$ and $[X,Y]=H$. 
We may choose a basis $E, E'$ in the abelian factor $\RR^2$ so that
$E$ is the grading element in $\fg$, while $E'$ annihilates the
abelian subalgebra $\fa\subset\fg_{-1}$ and acts as identity on $\fx$.
Then, the quintic component in $H_2^1(\fp_+,\fg)$ has weight $1$ for $E$ and $0$
for $E'$. As a representation of $\fsl(2,\RR)$, it is identified with the space
of quintic polynomials in $\RR[z,w]$. We list $G_0$-conjgacy classes
of nonzero elements with at least two-dimensional stabiliser:
\begin{center}\begin{tabular}{r|cl}
label & quintic & stabiliser \\
\hline
$N$ & $z^5$ & $\langle X,H-5E,E'\rangle$ \\
$IV$ & $z^4 w$ & $\langle H-3E,E'\rangle$ \\
$F$ & $z^3 w^2$ & $\langle H-E,E'\rangle$.
\end{tabular}\end{center}
\begin{lem}
The following is a one-to-one enumeration of
$G_0$-conjgacy classes of graded subalgebras $\fg_- \subset \u\fk\subset\fg$
such that $\dim\u\fk \ge 10$ and $\u\fk$ preserves a nonzero element in the quintic
component of the harmonic curvature module: 
\begin{center}\begin{tabular}{l|l}
label & $\u\fk_0$ \\
\hline
$N_3$ & $\langle X,H-5E,E'\rangle$   \\
$N_{2a}^\lambda$, $\lambda\in\RR\PP^1$ & $\langle X, \lambda_0(H-5E)+\lambda_1 E'\rangle$ \\
$N_{2b}$ & $\langle H-5E,E'\rangle$ \\
$IV_2$ & $\langle H-3E,E'\rangle$ \\
$F_2$ & $\langle H-E,E'\rangle$.
\end{tabular}\end{center}
In each case, $\u\fk = \fg_- + \u\fk_0$.
\end{lem}
\begin{proof}
The latter equality holds by Lemma \ref{lem:prol}, whence it 
remains
to find $G_0$-conjugacy classes of two-dimensional subalgebras $\u\fk_0$ in 
$\langle X, H-5E, E'\rangle \subset \fg_0$. 
Choosing an element
$Z \in \u\fk_0$ we have, up to conjugation, either $Z \in \langle H-5E,E'\rangle$ (semisimple)
or $Z=X$ (nilpotent). 
Thus, either $\u\fk_0 = \langle H-5E,E'\rangle$ or
it is spanned by $X$ and a one-dimensional subspace in $\langle H-5E,E'\rangle$.
\end{proof}

We thus have four discrete classes: $N_3, N_{2b}, IV_2, F_2$,
as well as a one-parameter family $N^\lambda_{2a}$, $\lambda\in\RR\PP^1$.
When computing the Kuranishi families and algebraic Cartan connections
for classes in the latter family, one needs to take care of the parameter $\lambda$.
We will first discuss the discrete classes, as the computations are straightforward
there. 

\subsection{Discrete cases}
Recall that the Kuranishi family for the
deformation problem associated with a given $\fg_-\subset\u\fk\subset\fg$
will be an algebraic subset  $M_{\u\fk} \subset H^{2,1}(\u\fk,\u\fk)^{\u\fk\cap\fh}$ 
together with a representing
map $M_{\u\fk} \to C^{2,1}(\u\fk,\u\fk)^{\u\fk\cap\fh}$ sending a point of $M_{\u\fk}$
to a suitable cochain $\phi$ satisfying the Maurer-Cartan condition. In the particular case of
$H^{3,2}(\u\fk,\u\fk)=0$, we have $M_{\u\fk} = H^{2,1}(\u\fk,\u\fk)$;
it will be the case for the four discrete classes. Indeed, we list the `Betti numbers'
$b^i_j = H^i_j(\u\fk,\u\fk)$:
\begin{center}\begin{tabular}{l|ll}
label & $b^2_j$, $j\ge 1$ & $b^3_j$, $j\ge 2$ \\
\hline
$N_3$ & $1,0,\dots$ & $0,\dots$ \\
$N_{2b}$ & $1,0,\dots$ & $0,\dots$ \\
$IV_2$ & $2,0,\dots$ & $0,\dots$ \\
$F_2$ & $2,0,\dots$ & $0,\dots$
\end{tabular}\end{center}
Thus, in all four cases the Kuranishi family is the entire affine space
of dimension $1$ or $2$.
Denoting the coordinates by $t$ for
$N_3,N_{2b}$ and by $t,s$ for $IV_2, F_2$,
we shall write down the representing map as a
deforming cocycle $\phi(t)$ or $\phi(t,s)$, a polynomial with
coefficients in $C^2(\u\fk,\u\fk)$. 

Let us introduce a convenient description of $\fg_-$. Recalling
$\fg_0 \simeq \fsl(2,\RR) \oplus \langle E,E'\rangle$, we have
identifications 
\begin{eqnarray*}
\fg_{-1} &\simeq& \RR^2[-1,0] \oplus \RR[-1,1] \\
\fg_{-2} &\simeq& \RR^2[-2,1] \\
\fg_{-3} &\simeq& (S^2\RR^2)[-3,1]
\end{eqnarray*}
where $\RR^d[\lambda,\lambda']$ denotes the irreducible $U(\fg_0)$-module
of rank $d$
on which $E$ acts with weight $\lambda$, and $E'$ with weight $\lambda'$. 
The structure equations are essentially determined
by the requirement that they be homomorphisms of $U(\fg_0)$-modules;
normalization of the intertwiners may be absorbed into the above
identifications. The $U(\fg_0)$-module
$\fg_{-1}^* \otimes \fg_{-3}^* \otimes \fg_{-3}$
contains a unique copy of $S^5\RR^2[1,0]$, identified with the quintic
component of $H^2(\fg_-,\fg)$. We let $\upsilon^N \in S^5\RR^2[1,0]$
be a non-zero vector annihilated by $X$, and set
$\upsilon^{IV} = Y\upsilon^N$, $\upsilon^F = Y^2\upsilon^N$. These
are the weight vectors in $S^5\RR^2[1,0]$ corresponding to the types
$N$, $IV$ and $F$. Using the $U(\fg_0)$-equivariant projections,
we may view them as maps $\Lambda^2\fg_-\to\fg_-$, and thus as
elements of $C^2(\u\fk,\u\fk)$ via the inclusion $\fg_-\subset\u\fk$. 

The representing map for $N_3$ and $N_{2b}$
may be immediately written as
$$
\phi(t) = t \upsilon^N.
$$
For types $IV_2$ and $F_2$,
we find a pair $\upsilon',\upsilon'' \in C^2(\u\fk,\u\fk)^{\u\fk\cap\fh}$
representing a basis for $H^2(\u\fk,\u\fk)$,
and sharing the same $U(\fh)$-weight as $\upsilon^{IV}$, resp. $\upsilon^{F}$.
For $IV_2$, all Richardson-Nijenhuis brackets involving $\upsilon',\upsilon''$
vanish and the representing map is
$$
\phi(t,s) = t\upsilon' + s\upsilon''.
$$
For $F_2$, there is an additional element $\upsilon''' \in C^2(\u\fk,\u\fk)^{\u\fk\cap\fh}$
such that $d\upsilon''' + \frac{1}{2}[\upsilon',\upsilon'']$, and
the representing map is
$$
\phi(t,s) = t\upsilon' + s\upsilon'' + st\upsilon'''.
$$ 
The other brackets are trivial.
Note that in both $IV_2$ and $F_2$, $\phi(t,s)$ is annihilated
by entire $\u\fk_0$.

We thus have cochains $\phi(t)$ or $\phi(t,s)$ defining
one- or two-parameter families of filtered Lie algebras
deforming $\u\fk$ for each of the four discrete classes.
The next step is to compute the regular, normal algebraic
Cartan connection for every such deformation: this
is done in terms of linear algebra over $\RR(t)$ or $\RR(t,s)$
and yields a map $\bar\omega:\fk \to \fg$ whose coefficients
are polynomials in $t$ or $t,s$. One may then find the harmonic
curvature, which turns out to be contained in the quintic component
for all values of the parameters $t,s$. Hence, $M_{\u\fk}' = M_{\u\fk}$
in all four classes.

Further, one computes the
dimension of the symmetry algebra as in Lemma \ref{lem:sym}.
This involves finding kernels of certain matrices over
$\RR(t)$ or $\RR(t,s)$, and thus gives the symmetry dimension
for the \emph{generic} member of the deformation family
in each class. The results follow:
\begin{center}\begin{tabular}{r|cccc}
label & $N_3$ & $N_{2b}$ & $IV_2$ & $F_2$ \\
\hline
generic $d$ & 11 & 11 & 10 & 10.
\end{tabular}
\end{center}
It turns out that the generic members of
$N_3$, $IV_2$ and $F_2$ are the actual symmetry algebras
of the Monge geometries they define -- 
hence, $M_{\u\fk}''$ is a dense Zariski open subset of $M_{\u\fk}$.
On the other hand, by upper-semicontinuity,
we have $d\ge 11$ in type $N_{2b}$ so that the actual symmetry algebra is 
always larger than $\fk$ in that type. Accordingly,
$M_{\u\fk}'' = \emptyset$ for $N_{2b}$ and the entire family may be discarded.

We are left with three classes: $N_3$, $IV_2$ and $F_2$.
The next step is to identify the action of $G_0^{\u\fk} \subset G_0$ on
$M_{\u\fk}$ and describe the quotient $M''_{\u\fk}/G_0^{\u\fk}$. 
Recall that $G_0 \simeq \SL(2,\RR) \times (\RR^{\times})^2$.
Let $T \subset G_0$ denote the maximal torus with Lie algebra $\fh\subset\fg_0$.
The subalgebra $\u\fk \subset \fg$ is always $T$-invariant,
so that $T \subset G_0^{\u\fk}$.
Furthermore,
all constructions related to the Kuranishi family are $T$-equivariant, so that
$M_{\u\fk}$ and $M_{\u\fk}''$ are $T$-invariant subsets in $H^{2,1}(\u\fk,\u\fk)$,
and the representing map $\phi : M_{\u\fk} \to C^{2,1}(\u\fk,\u\fk)^{\u\fk\cap\fh}$
is $T$-equivariant. 
It follows that the $T$-action on $M_{\u\fk}$ factors through
the multiplicative action of $\RR^\times$ by means of the grading element
(recall that $\upsilon$ and $\upsilon'$ have the same weight for $U(\fh)$,
hence for $T$). In coordinates, both $t$ and $s$ transform under $\RR^\times$ with weight $1$.
The identity component $G_0^{\u\fk,0} \subset G_0^{\u\fk}$
and the associated quotient is then easy to describe:
\begin{center}\begin{tabular}{r|ccc}
label & $N_3$ & $IV_2$ & $F_2$ \\
\hline
$G_0^{\u\fk,0}$ & $T\ltimes\exp\langle X\rangle$ & $T$ & $T$ \\
$M_{\u\fk} / G_0^{\u\fk,0}$ & $\{0,*\}$ & $\{0\}\cup\RR\PP^1$ & $\{0\}\cup\RR\PP^1$ \end{tabular} 
\end{center}
where $0$ is the trivial deformation (hence flat), and $*$ is the unique non-flat model in $N_3$.

This is already the minimal possible quotient for $N_3$, so that
there is precisely one homogeneous model of type $N$ with $3$-dimensional
isotropy (up to equivalence). In the remaining types $IV_2$ and $F_2$, we 
note that $G_0^{\u\fk}$ is contained in the normaliser $N_{G_0}T$ of the torus $T$ in $G_0$,
while $G_0^{\u\fk,0}$ is simply its centraliser, $T$. Now, $N_{G_0}T/T \simeq \ZZ_2$,
the Weyl group of $\SL(2,\RR)$: its generator acts on $\fh$
sending $H$ to $-H$ and preserving both $E$ and $E'$.
Since $\u\fk_0$ is generated by $E'$ and a \emph{combination} of $H$ and $E$, it follows that
$G_0^{\u\fk}$ must actually preserve $H$, and thus coincides with the identity component
$T = G_0^{\u\fk}$ of $N_{G_0}T$.
An explicit computation of the algebraic Cartan connection shows that
the harmonic curvature of the geometry defined by $\phi(s,t)$ is proportional
to $s$, whence $s=0$ is the unique point on $\RR\PP^1 = (M_{\u\fk}/G_0^{\u\fk,0}) \setminus \{0\}$
corresponding to a flat Monge geometry.
Finally then, $M''_{\u\fk}/G_0{\u\fk,0}$ embeds onto a dense open in $\RR = \RR\PP^1 \setminus \{s=0\}$
parameterised by $t/s$.

\subsection{Continuous family case}
Let us now focus on the family $N^\lambda_{2a}$, $\lambda\in\RR\PP^1$.
Geometrically, one views the family of graded subalgebras $\u\fk_\lambda \subset \fg$
as a sub-bundle of the trivial vector bundle $\RR\PP^1 \times \fg$ over $\RR\PP^1$.
Correspondingly, the cochain complexes $C^\bullet(\u\fk_\lambda,\u\fk_\lambda)$
give a complex of vector bundles over $\RR\PP^1$. The differential and the
Nijenhuis-Richardson bracket become homomorphisms of vector bundles,
and in particular their fibre-wise images and kernels may have non-constant dimension.
In the abstract deformation setting, we then have a bundle of DGLAs $\bigcup_\lambda\fL_\lambda
\to\RR\PP^1$, and we wish to construct Kuranishi families fibrewise.
It is intuitively clear that, since all the data involved are of finite type,
there should be a finite subset $\Lambda \subset \RR\PP^1$ such that
the calculations may be performed:
\begin{itemize}
\item separately for each $\lambda\in\Lambda$, and
\item working over the field $\RR(\lambda)$ for the complement $\RR\PP^1\setminus\Lambda$. 
\end{itemize}
Indeed, the following provides the proper way to deal with $N_{2a}^\lambda$.
\begin{lem}\label{lem:family}
Identify $\RR\PP^1$ with the projectivization of $\fa=\langle H-5E,E'\rangle$, so that
$\lambda = (\lambda_0,\lambda_1)$ corresponds to the line spanned by
$\lambda_0(H-5E)+\lambda_1 E'$. 
Set $\fn = \fg_- \oplus \langle X \rangle$, so that  
$\u\fk_\lambda = \fn \oplus \lambda$ is the graded subalgebra of $\fg$
corresponding to $\lambda\in\RR\PP^1$. Let $\fL_\lambda = \fL(\u\fk_{\lambda})$ 
and let $\fL_\lambda^\fa \subset \fL_\lambda$
denote the sub-DGLA annihilated by $\lambda \subset \u\fk_\lambda$. Recall that
$\fL_\lambda^\fa$ and $\fL_\lambda$ are equipped with an additional grading
induced by that on $\u\fk_\lambda$.

Let $\Sigma \subset \fa^* \setminus \{0\}$ be the set
of nonzero weights of the $U(\fa)$-module $C^{\bullet,1}(\fn,\fn) \oplus C^{\bullet,1}(\fn)$,
and let $\Lambda \subset \RR\PP^1$
be the union of the zero-loci of the linear forms in $\Sigma$.
Then:\begin{enumerate}
\item 
$\bigcup_{\lambda\in\RR\PP^1\setminus\Lambda} \fL^\fa_\lambda$
forms a bundle of DGLA over $\RR\PP^1\setminus\Lambda$;
\item there is a DGLA
$\fM$ with additional grading, and an 
isomorphism
$$
\bigcup_{\lambda\in\RR\PP^1\setminus\Lambda} \fL^\fa_\lambda
\simeq
\fM \times (\RR\PP^1\setminus\Lambda)
$$
of bundles of DGLA over $\RR\PP^1\setminus\Lambda$,
compatible with the additional gradings.
\end{enumerate}
\end{lem}
\begin{proof}
The result is more naturally cast in (real) algebro-geometric terms,
where the family $(\u\fk_\lambda)_\lambda$ is identified
with a coherent sheaf $\u\sK$ of graded Lie algebras,
and $(\fL_\lambda)_\lambda$ with a coherent sheaf $\sL$ of DGLA on $\RR\PP^1$.
Both are locally free as sheaves of $\OO_{\RR\PP^1}$-modules,
in particular $$\u\sK \simeq \fn\otimes\OO_{\RR\PP^1} \oplus \OO_{\RR\PP^1}(-1)$$
with the second summand in graded degree zero.
Now, $(\fL^\fa_\lambda)_\lambda$ corresponds to the kernel in the
exact sequence $$ 0 \to \sL^\fa \to \sL \to \sL \otimes \OO_{\RR\PP^1}(1) = \sL(1)$$
where the rightmost map is defined by the action
of $\OO_{\RR\PP^1}(-1) \subset \u\sK$ on $\sL$. Restricting over
$\RR\PP^1\setminus\Lambda$, the inclusion $\sL^\fa \to \sL$ becomes split,
so that in particular $\sL^\fa$ is the sheaf of sections of a vector sub-bundle
$\bigcup_{\lambda\notin\Lambda}\fL^\fa_\lambda \subset \bigcup_{\lambda\notin\Lambda}\fL_\lambda$.
Let us choose a local trivialization
$$
\u\sK|_{\RR\PP^1\setminus\Lambda} \simeq (\fn \oplus \RR) \otimes \OO_{\RR\PP^1\setminus\Lambda}
$$
as graded locally free sheaves. There is an induced
trivialization 
$$
\sL|_{\RR\PP^1\setminus\Lambda} \simeq \tilde\fM
\otimes \OO_{\RR\PP^1\setminus\Lambda},\quad
\tilde\fM=
C^{\bullet,1}(\fn,\fn) \oplus C^{\bullet-1,1}(\fn,\fn) \oplus 
C^{\bullet,1}(\fn)  \oplus C^{\bullet-1,1}(\fn)
$$
as doubly graded locally free sheaves. We claim that
there is a graded subspace $\fM \subset \tilde\fM$ equipped with a compatible
structure of a DGLA, 
such that the above restricts to an isomorphism
$$
\sL^\fa_{\RR\PP^1\setminus\Lambda} \simeq \fM \otimes \OO_{\RR\PP^1\setminus\Lambda}
$$
of DGLA with additional grading. In fact $\fM$ is precisely the subspace
$\tilde\fM^\fa \subset \tilde\fM$ annihilated by $\fa$. Indeed, it is clear that
$\tilde\fM^\fa \otimes \OO_{\RR\PP^1\setminus\Lambda}$ 
is contained in the image of $\sL^\fa_{\RR\PP^1\setminus\Lambda}$. Furthermore,
the top row in the diagram
$$\begin{CD}
0 @>>>
\tilde\fM^\fa \otimes \OO_{\RR\PP^1\setminus\Lambda} @>>>
\tilde\fM \otimes \OO_{\RR\PP^1\setminus\Lambda} @>>> 
\tilde\fM \otimes \OO_{\RR\PP^1\setminus\Lambda}(1) \\
@. @VVV @| @| \\
0 @>>> \sL^\fa|_{\RR\PP^1\setminus\Lambda} @>>> \sL|_{\RR\PP^1\setminus\Lambda} @>>> \sL(1)|_{\RR\PP^1\setminus\Lambda} 
\end{CD}$$
is exact, since the only elements in $\tilde\fM$ annihilated by
a subspace of $\fa$ corresponding to a point of $\RR\PP^1\setminus\Lambda$
are those annihilated by entire $\fa$. Hence, the leftmost vertical arrow is
an isomorphism. Finally, an inspection of the DGLA structure on $\sL^\fa$
shows that the induced family of DGLA structures on $\fM = \tilde\fM^\fa$ is constant.
\end{proof}

As a consequence, it is enough to compute the Kuranishi family 
for the finitely many deformations described by points of $\Lambda$,
and for the generic point of $\RR\PP^1$.
The set $\Lambda$ includes $\infty = (0:1)$, so that its
complement is an open subset of the affine line parameterised by $\lambda=\lambda_1/\lambda_0$.
The computations for
$N_{2a}^\lambda$, $\lambda\in\Lambda$ are performed as previously for the
discrete classes. Then, for generic $\lambda$, we compute the Kuranishi
family $M_\fM$ for the DGLA $\fM$ as in Lemma \ref{lem:family}. Considering
$\bigcup_\lambda C^{2,1}(\u\fk_\lambda,\u\fk_\lambda)$ as a vector bundle
over $\RR\PP^1$, the isomorphism of Lemma \ref{lem:family} gives the representing
section
$$
\Phi : M_\fM \times (\RR\PP^1\setminus\Lambda) \to \bigcup_{\lambda\notin\Lambda} C^{2,1}(\u\fk_\lambda,
\u\fk_\lambda)
$$
as an algebraic homomorphism of vector bundles over $\RR\PP^1\setminus\Lambda$. In fact,
$M_\fM$ is an affine line; we parameterise it with a coordinate $t$.
Since $\RR\PP^1\setminus\Lambda$ is contained in the affine line $\RR\PP^1\setminus\{\infty\}$,
the bundle $\bigcup_{\lambda\notin\Lambda} \u\fk_\lambda$ is trivial as
a vector bundle, and isomorphic to $(\fn\oplus\RR) \times (\RR\PP^1\setminus\Lambda)$. 
Using this trivialization, one checks that
the representing map, written as
$$
\Phi : \RR \times (\RR\PP^1\setminus\Lambda) \to (\Lambda^2(\fn\oplus\RR)^* \otimes (\fn\oplus\RR))
\times (\RR\PP^1\setminus\Lambda)
$$
is simply the constant lift of
$$ \phi : \RR \to \Lambda^2\fn^* \otimes\fn,\quad \phi(t) = t\upsilon^N. $$
Still, the Lie algebra structure in the fibres of $\bigcup_{\lambda\notin\Lambda}\u\fk_\lambda$
depends algebraically on $\lambda$, and so does the deformed structure. One then
computes a regular, normal algebraic Cartan connection as well as the symmetry
dimension for the \emph{generic} member in terms of linear algebra over $\RR(t,\lambda)$.
It turns out that the generic symmetry algebra is $11$-dimensional. Since
the dimension function 
$ d : \RR \times (\RR\PP^1\setminus\Lambda) \to \ZZ$
is Zariski upper-semicontinuous, it follows that $d\ge 11$ for all members of $N_{2a}^\lambda$, $\lambda\notin\Lambda$, whence the entire family may be discarded.

Turning to the finitely many models parameterised by $\Lambda$,
it turns out that the only value of $\lambda$ with 
nonempty $M_{\u\fk}''$ is $\lambda=\infty$. We thus omit the calculations for the other
elements of $\Lambda$, and focus on $N_{2a}^\infty$. The underlying graded subalgebra
of $\fg$ is $ \u\fk = \fg_- \oplus \langle X,E' \rangle$. 
The interesting `Betti numbers' read $b^2_1 = 3$, $b^2_2 = 1$, $b^2_i = 0$ for $i\ge 3$.
We introduce coordinates $t_1,t_2,t_3$ on $H^2_1(\u\fk,\u\fk)$ and
$s$ on $H^2_2(\u\fk,\u\fk)$.
It turns out that this time $H^{3,2}(\u\fk,\u\fk)$ is non-trivial and
the Kuranishi family $M_{\u\fk} \subset H^{2,1}(\u\fk,\u\fk)$ is the union
of two irreducible components: the hyperplane $t_3=0$ and
the line $t_1=t_2=s=0$. Computing the algebraic Cartan connection and
harmonic curvature for the generic points of the two components,
one checks that deformations parameterised by the one-dimensional component 
give rise to flat Monge geometries. On the other hand, the value of the
dimension function $d$ at the generic point of the three-dimensional component
is $d=10$. We also have $M'_{\u\fk}=M_{\u\fk}$ so that $M''_{\u\fk}$ is a dense
open subset of the three-dimensional irreducible component parameterised by $t_1,t_2,s$.
The curvature of our representative connection is harmonic and proportional to $t_1$,
so that in particular $M''_{\u\fk}$ is contained in the subset $t_1\neq0$.

We now need to pass to the quotient by the action of $G_0^{\u\fk}$. As before,
one easily describes $G_0^{\u\fk,0} = T \times \exp\langle X\rangle$. The
Kuranishi family and representing map are $T$-equivariant, and the weights
of $t_1,t_2,s$ under the action of $T$ are given by:
\begin{eqnarray*}
e^{aH} (t_1,t_2,s) &=& (e^{5a} t_1, e^a t_2, e^{2a}s)\\
e^{aE} (t_1,t_2,s) &=& (e^a t_1, e^a t_2, e^{2a}s)\\
e^{aE'} (t_1,t_2,s) &=& (t_1,t_2,s).
\end{eqnarray*}
One checks by explicit calculation that the remaining unipotent factor $\exp \langle X \rangle$
acts trivially on the three-dimensional component of $M_{\u\fk}$, i.e.
that the action of $X$ on the cochain parameterised by $t_1,t_2,s$
corresponds to the gauge action of a suitable element of $C^{1,1}(\u\fk,\u\fk)$.
Now, $M''_{\u\fk}$ is contained in the 
complement of $t_1=0$, whose image in $\RR^3/G_0^{\u\fk,0}$ may
be described by normalising $t_1=1$ and taking the quotient
of $\RR^2$, parameterised by $t_2,s$, by the rank one sub-torus $\RR^\times \subset T$
generated by $H-5E$. This latter torus acts on $t_2$ with weight $4$, and on
$s$ with weight $8$.
It follows that $M''_{\u\fk}/G_0^{\u\fk,0}$ is an open subset
of $\RR^2/\RR^{\times} \simeq \{*\} \cup (-\infty,\infty)\times\{\pm1\} \cup \{\pm\infty\}$,
where $*$ is the singleton orbit of $t_2=s=0$,
the endpoints $\pm\infty$ parameterise half-lines $t_2=0$, $s>0$ and
$t_2=0$, $s<0$, while $(-\infty,\infty)\times\{\pm1\}$ 
is the set of half-parabolas $s =\alpha t_2^2$, $t_2 = \epsilon$ with 
$\alpha\in\RR$, $\epsilon=\pm1$.
The point $*$,
corresponding to $t_1=1$, $t_2=s=0$, gives rise to a geometry with
$11$-dimensional symmetry of type $N_3$, thus $M''_{\u\fk}/G_0^{\u\fk,0} \subset 
(-\infty,\infty) \times \{\pm1\} \cup \{\pm\infty\}$. It is convenient to visualise
the latter space as two parallel copies of the real line, sharing the endpoints $\pm\infty$.
Finally, it is easy to observe that the full subgroup of $\SL(2,\RR)$ preserving
both $H$ and $\langle X\rangle$ is connected, and so is thus $G_0^{\u\fk}$. Hence,
$M''_{\u\fk}/G_0^{\u\fk}$ is a dense open subset of $(-\infty,\infty) \times \{\pm1\}
 \cup \{\pm\infty\}$, parameterised by $s/t_2^2$ and the sign of $t_2$.
An explicit computation of the symmetry algebra dimension for the endpoints $\pm\infty$
shows that they do belong to $M''_{\u\fk}/G_0^{\u\fk}$.

\subsection{Final result}
We thus arrive at the final list of homogeneous models. We express the structure
equations in terms of the exterior derivatives of the dual basis of left-invariant forms on a Lie group 
$K$ integrating $\fk$. The vertical distribution on $K \to K/L$ is the annihilator of
$\theta_1,\dots,\theta_8$, while the preimage in $TK$ of the Monge distribution is the annihilator
of $\theta_1,\dots,\theta_5$.
\begin{itemize}
\item Type $N_3$.
 \begin{eqnarray*}
 d\theta^{1} & = &   -2 \theta^{1}\wedge\theta^{9}  -2 \theta^{1}\wedge\theta^{11}  -2 \theta^{4}\wedge\theta^{6}\\
 d\theta^{2} & = &   - \theta^{1}\wedge\theta^{10}  -2 \theta^{2}\wedge\theta^{11}  - \theta^{4}\wedge\theta^{7}  - \theta^{5}\wedge\theta^{6}\\
 d\theta^{3} & = &   - \theta^{1}\wedge\theta^{6}  -2 \theta^{2}\wedge\theta^{10} + 2 \theta^{3}\wedge\theta^{9}  -2 \theta^{3}\wedge\theta^{11}  -2 \theta^{5}\wedge\theta^{7}\\
 d\theta^{4} & = &   -6 \theta^{4}\wedge\theta^{9}  -2 \theta^{4}\wedge\theta^{11} +  \theta^{6}\wedge\theta^{8}\\
 d\theta^{5} & = &   - \theta^{4}\wedge\theta^{10}  -4 \theta^{5}\wedge\theta^{9}  -2 \theta^{5}\wedge\theta^{11} +  \theta^{7}\wedge\theta^{8}\\
 d\theta^{6} & = &   4 \theta^{6}\wedge\theta^{9}\\
 d\theta^{7} & = &   - \theta^{6}\wedge\theta^{10} + 6 \theta^{7}\wedge\theta^{9}\\
 d\theta^{8} & = &   -10 \theta^{8}\wedge\theta^{9}  -2 \theta^{8}\wedge\theta^{11}\\
 d\theta^{9} & = & 0\\
 d\theta^{10} & = &   -2 \theta^{9}\wedge\theta^{10}\\
 d\theta^{11} & = & 0
\end{eqnarray*}
\item Type $N_{2a}^\infty$, $\epsilon=\pm1$, $\alpha\in\RR$ generic (interior).
 \begin{eqnarray*}
 d\theta^{1} & = &    \epsilon\theta^{1}\wedge\theta^{6}  -2 \theta^{1}\wedge\theta^{10}  -2 \theta^{4}\wedge\theta^{6}\\
 d\theta^{2} & = &   - \theta^{1}\wedge\theta^{9}  -2 \theta^{2}\wedge\theta^{10}  - \theta^{4}\wedge\theta^{7}  - \theta^{5}\wedge\theta^{6}\\
 d\theta^{3} & = &    \theta^{1}\wedge\theta^{6}  -2 \theta^{2}\wedge\theta^{9}  -\epsilon \theta^{3}\wedge\theta^{6}  -2 \theta^{3}\wedge\theta^{10}  -2 \theta^{5}\wedge\theta^{7}\\
 d\theta^{4} & = &   \alpha \theta^{1}\wedge\theta^{6}  -2 \theta^{4}\wedge\theta^{10} +  \theta^{6}\wedge\theta^{8}\\
 d\theta^{5} & = &   \alpha \theta^{2}\wedge\theta^{6}  - \theta^{4}\wedge\theta^{9}  - \epsilon\theta^{5}\wedge\theta^{6}  -2 \theta^{5}\wedge\theta^{10} +  \theta^{7}\wedge\theta^{8}\\
 d\theta^{6} & = & 0\\
 d\theta^{7} & = &   - \theta^{6}\wedge\theta^{9}\\
 d\theta^{8} & = &   2\alpha \theta^{4}\wedge\theta^{6} +  \epsilon\theta^{6}\wedge\theta^{8}  -2 \theta^{8}\wedge\theta^{10}\\
 d\theta^{9} & = &   \alpha \theta^{6}\wedge\theta^{7} +  \epsilon\theta^{6}\wedge\theta^{9}\\
 d\theta^{10} & = & 0
\end{eqnarray*}
\item Type $N_{2a}^\infty$, $\sigma=\pm1$ (boundary).
 \begin{eqnarray*}
 d\theta^{1} & = &   -2 \theta^{1}\wedge\theta^{10}  -2 \theta^{4}\wedge\theta^{6}\\
 d\theta^{2} & = &   - \theta^{1}\wedge\theta^{9}  -2 \theta^{2}\wedge\theta^{10}  - \theta^{4}\wedge\theta^{7}  - \theta^{5}\wedge\theta^{6}\\
 d\theta^{3} & = &    \theta^{1}\wedge\theta^{6}  -2 \theta^{2}\wedge\theta^{9}  -2 \theta^{3}\wedge\theta^{10}  -2 \theta^{5}\wedge\theta^{7}\\
 d\theta^{4} & = &   \sigma \theta^{1}\wedge\theta^{6}  -2 \theta^{4}\wedge\theta^{10} +  \theta^{6}\wedge\theta^{8}\\
 d\theta^{5} & = &   \sigma \theta^{2}\wedge\theta^{6}  - \theta^{4}\wedge\theta^{9}  -2 \theta^{5}\wedge\theta^{10} +  \theta^{7}\wedge\theta^{8}\\
 d\theta^{6} & = & 0\\
 d\theta^{7} & = &   - \theta^{6}\wedge\theta^{9}\\
 d\theta^{8} & = &   2\sigma \theta^{4}\wedge\theta^{6}  -2 \theta^{8}\wedge\theta^{10}\\
 d\theta^{9} & = &   \sigma \theta^{6}\wedge\theta^{7}\\
 d\theta^{10} & = & 0
\end{eqnarray*}
\item Type $IV_2$, $\alpha\in\RR$ generic.
 \begin{eqnarray*}
 d\theta^{1} & = &   -2 \theta^{1}\wedge\theta^{9}  -2 \theta^{1}\wedge\theta^{10}  -2 \theta^{4}\wedge\theta^{6}\\
 d\theta^{2} & = &   (2+2\alpha) \theta^{1}\wedge\theta^{6}  -2 \theta^{2}\wedge\theta^{10}  - \theta^{4}\wedge\theta^{7}  - \theta^{5}\wedge\theta^{6}\\
 d\theta^{3} & = &   2\alpha \theta^{2}\wedge\theta^{6} + 2 \theta^{3}\wedge\theta^{9}  -2 \theta^{3}\wedge\theta^{10}  -2 \theta^{5}\wedge\theta^{7}\\
 d\theta^{4} & = &   -4 \theta^{4}\wedge\theta^{9}  -2 \theta^{4}\wedge\theta^{10} +  \theta^{6}\wedge\theta^{8}\\
 d\theta^{5} & = &   \alpha \theta^{4}\wedge\theta^{6}  -2 \theta^{5}\wedge\theta^{9}  -2 \theta^{5}\wedge\theta^{10} +  \theta^{7}\wedge\theta^{8}\\
 d\theta^{6} & = &   2 \theta^{6}\wedge\theta^{9}\\
 d\theta^{7} & = &   4 \theta^{7}\wedge\theta^{9}\\
 d\theta^{8} & = &   -6 \theta^{8}\wedge\theta^{9}  -2 \theta^{8}\wedge\theta^{10}\\
 d\theta^{9} & = & 0\\
 d\theta^{10} & = & 0
\end{eqnarray*}
\item Type $F_2$, $\alpha\in\RR$ generic.
 \begin{eqnarray*}
 d\theta^{1} & = &   (1+2\alpha) \theta^{1}\wedge\theta^{6}  -2 \theta^{1}\wedge\theta^{9}  -2 \theta^{1}\wedge\theta^{10}  -2 \theta^{4}\wedge\theta^{6}\\
 d\theta^{2} & = &   \alpha \theta^{1}\wedge\theta^{7}  -2 \theta^{2}\wedge\theta^{10}  - \theta^{4}\wedge\theta^{7}  - \theta^{5}\wedge\theta^{6}\\
 d\theta^{3} & = &   2 \theta^{3}\wedge\theta^{9}  -2 \theta^{3}\wedge\theta^{10}  -2 \theta^{5}\wedge\theta^{7}\\
 d\theta^{4} & = &   \alpha \theta^{1}\wedge\theta^{6}  -2 \theta^{4}\wedge\theta^{9}  -2 \theta^{4}\wedge\theta^{10} +  \theta^{6}\wedge\theta^{8}\\
 d\theta^{5} & = &   2\alpha \theta^{5}\wedge\theta^{6}  -2 \theta^{5}\wedge\theta^{10} +  \theta^{7}\wedge\theta^{8}\\
 d\theta^{6} & = & 0\\
 d\theta^{7} & = &   2\alpha \theta^{6}\wedge\theta^{7} + 2 \theta^{7}\wedge\theta^{9}\\
 d\theta^{8} & = &   -4\alpha \theta^{6}\wedge\theta^{8}  -2 \theta^{8}\wedge\theta^{9}  -2 \theta^{8}\wedge\theta^{10}\\
 d\theta^{9} & = & 0\\
 d\theta^{10} & = & 0
\end{eqnarray*}
 \end{itemize}

\section{Comparison to the work of Anderson and Nurowski}
\subsection{}
We may now compare our results to those of Anderson and  Nurowski in~\cite{AN}. As we have
already remarked, the geometries they consider are restricted to type $N$,
but the classification is carried through all the way down to simply-transitive
models. Hence, their list intersects ours in the $11$-dimensional
model of type $N_3$, as well as in the one-parameter family of 
$10$-dimensional models $N^\infty_{2a}$. It is indeed possible to
find an explicit equivalence between the realisations written down
in the two papers.

\subsection{} Naturally, the possibility of such comparison prompts the question about
the advantages and disadvantages of each of the two methods, i.e. deformation
theory in our case and Cartan reduction in~\cite{AN}. We shall only contrast the
two techniques in a couple of aspects. First, we note that Cartan reduction,
if properly performed, gives a much more efficient algorithm in the sense
that it necessarily produces a complete list of inequivalent models. Furthermore,
since assumptions on the possible normalisations of structure functions are
built into the very core of the procedure, it is straightforward to restrict
the classification to a given curvature type (e.g. type $N$, as Anderson and Nurowski
have done), with a guarantee that no examples of other types are produced. 

Our method
gives no such control over the curvature; while we do use assumptions on the curvature
type to produce a list of classes of isotropy subalgebras $\u\fk_0 \subset \fg_0$,
it may well happen that a particular component of the Kuranishi space deforms the
flat algebra into a different curvature type that intended (this had been observed
for some of the $1$-transitive models). We are thus in general forced to filter our 
raw list of models. Likewise, when parameterising deformations belonging to a given
component of a Kuranishi space for some $\u\fk$, we ought to carefully identify (1) intersections
with other components, and (2) points that should be glued to a Kuranishi space
for a larger $\u\fk' \supset \u\fk$. For instance, comparing our family $N_{2a}^\infty$
to the one found by Anderson and Nurowski, we see that certain special values of
the parameter $\alpha$ should be excluded as defining models with a larger symmetry algebra
(recall that we compute the dimension of the symmetry algebra only at the generic point
of each irreducible component). 

By design, these problems do not arise in the
approach based on Cartan reduction. However, a careful reading
of~\cite{AN} shows that achieving such efficiency requires a great deal
of experience and technical skill. This is in contrast to the
deformation-theoretic approach, which comes with the advantage of
using standard cohomological tools susceptible to far-reaching automation,
and may be applied without a deep understanding of the invariants of a geometry
in question. In future applications one might perhaps use the deformation-theoretic
method algorithmically to produce an early survey of the terrain, only to be elaborated by means
of Cartan reduction.

\subsection{}
Let us finally
remark that our method may be in principle applied also to produce $1$-transitive homogeneous
models of the $C_3$ Monge geometry. The computations in that
case become too large to handle without first developing more subtle algorithms.
To give the reader a taste of the situation, let us mention that
the graded subalgebras $\u\fk \subset \fg$
in this case include the \emph{nilpotent} algebra $\fg_- \oplus \langle X\rangle$,
admitting a $10$-dimensional $H^{2,1}(\u\fk,\u\fk)$ and a Kuranishi family
with 27 irreducible components. Interestingly enough, all those components, and indeed
all the irreducible components of all Kuranishi spaces we have encountered in this
project, turn out to be rational.

\bibliographystyle{plain}
\bibliography{c3}

\end{document}